\documentclass[12pt]{article}

\usepackage{amsmath}  
\usepackage{amscd}
\usepackage{eucal}
\usepackage{amssymb}
\usepackage{epic}
\usepackage{graphicx}
\graphicspath{{figures/}}
\usepackage{color}
\usepackage{mathrsfs}
\usepackage{multirow}
\usepackage{float}
\usepackage{subfig}
\usepackage{amsthm}
\usepackage{amssymb}
\usepackage{amsbsy}
\usepackage{extarrows}
\usepackage{enumerate}
\usepackage{verbatim}
\usepackage{nicefrac}
\usepackage{amsfonts}
\usepackage{mathrsfs}   
\usepackage{wrapfig,epic,eepic,pict2e}
\usepackage{tikz}
\usepackage{xparse}

\usepackage[only,llbracket,rrbracket,llparenthesis,rrparenthesis]{stmaryrd}
\usepackage{accsupp}

\numberwithin{equation}{section}
\numberwithin{table}{section}

\theoremstyle{plain}
\newtheorem{theorem}{Theorem}[section]
\newtheorem{lemma}{Lemma}[section]

\newtheorem{rmk}{Remark}[section]

\theoremstyle{remark}

\newcommand{\Eh}{\mathscr{E}_h}
\newcommand{\TT}{\mathscr{T}}

\newcommand{\Voh}{\mathring{V}_h} 
\newcommand{\Zoh}{\mathring{Z}_h} 
\def\norm#1#2{\left\| #1 \right\|_{#2}} 
\newcommand{\avephio}{\overline{\phi}_0}  
\newcommand{\dtau}{\delta_\tau}

\NewDocumentCommand{\dgal}{sO{}m}{%
  \IfBooleanTF{#1}
    {\dgalext{#3}}
    {\dgalx[#2]{#3}}%
}

\NewDocumentCommand{\dgalext}{m}{%
  \sbox0{%
    \mathsurround=0pt 
    $\left\{\vphantom{#1}\right.\kern-\nulldelimiterspace$%
  }%
  \sbox2{\{}%
  \ifdim\ht0=\ht2
    \{\kern-.625\wd2 \{#1\}\kern-.625\wd2 \}%
  \else
    \left\{\kern-.7\wd0\left\{#1\right\}\kern-.7\wd0\right\}%
  \fi
}

\NewDocumentCommand{\dgalx}{om}{%
  \sbox0{\mathsurround=0pt$#1\{$}%
  \sbox2{\{}%
  \ifdim\ht0=\ht2
    \{\kern-.625\wd2 \{#2\}\kern-.625\wd2 \}%
  \else
    \mathopen{#1\{\kern-.7\wd0 #1\{}
    #2
    \mathclose{#1\}\kern-.7\wd0 #1\}}
  \fi
}

\def\forme(#1,#2){\left<#1,#2\right>}
\newcommand{\iprd}[2]{\left( #1 , #2 \right)} 
\newcommand{\aIPh}[2]{a_h^{IP}\left( #1 , #2 \right)}

\newcommand{\aiprd}[2]{a\!\left( #1 , #2 \right)}

\newcommand{\phih}{\phi_h}
\newcommand{\muh}{\mu_h}
\newcommand{\psih}{\psi_h}
\newcommand{\nuh}{\nu_h}
\newcommand{\varphih}{\varphi_h}
\newcommand{\err}{\mathbbold{e}}
\newcommand{\errP}{\mathbbold{e}_P}
\newcommand{\errR}{\mathbbold{e}_R}
\newcommand{\errh}{\mathbbold{e}_h}

\DeclareSymbolFont{bbold}{U}{bbold}{m}{n}
\DeclareSymbolFontAlphabet{\mathbbold}{bbold}

\makeatletter

\newcommand{\Rmnum}[1]{\expandafter\@slowromancap\romannumeral #1@}

\makeatother
\definecolor{listinggray}{gray}{0.9}
\definecolor{lbcolor}{rgb}{0.9,0.9,0.9}
\definecolor{Darkgreen}{rgb}{0.0,0.39,0.0}
\date{\today}

\title{A C$^0$ Interior Penalty Method for the Phase Field Crystal Equation\thanks{The work of the second author
 was supported in part by the National Science Foundation under Grant No. DMS-1520862.}}

	\author{
Amanda E. Diegel\thanks{Department of Mathematics and Statistics, Mississippi State University, Mississippi State, MS 39762 (adiegel@math.msstate.edu)},
	\and
Natasha Sharma\thanks{Department of Mathematical Science, University of Texas at El Paso, El Paso, TX 79968 (nssharma@utep.edu)}}

\begin{document}

\maketitle

\begin{abstract}
We present a C$^0$ interior penalty finite element method for the sixth-order phase field crystal equation. We demonstrate that the numerical scheme is uniquely solvable, unconditionally energy stable, and convergent. We remark that the novelty of this paper lies in the fact that this is the first C$^0$ interior penalty finite element method developed for the phase field crystal equation. Additionally, the error analysis presented develops a detailed methodology for analyzing time dependent problems utilizing the C$^0$ interior penalty method. We furthermore benchmark our method against numerical experiments previously established in the literature.
\end{abstract}


\section{Introduction}\label{sec:intro}
Phase field crystal (PFC) methodolgy is emerging as a popular means of modeling many important phenomena observed in materials science. Examples include grain growth, dendritic and eutectic solidification, epitaxial growth, and more. We refer the interested reader to the review paper \cite{PDEetal:07:PFCreview} for more details including references to the examples above. Due to the growing number of applications of the PFC model and its variations, there is an interest in developing accurate and stable numerical schemes. Indeed, much progress has already been accomplished within the finite difference and Fourier spectral framework. The goal of this paper, however, is to present a finite element approach to the PFC model which is unconditionally energy stable, uniquely solvable, and convergent.

Let $ \Omega \subset \mathbb{R}^2$ be an open polygonal domain and consider the dimensionless energy of the form \cite{WWL:09:PFC}
\begin{align}\label{eq:pfc-energy}
E(\phi) = \int_{\Omega} \left\{\frac{1}{4}\phi^4 + \frac{1-\epsilon}{2}\phi^2 - |\nabla \phi|^2 +\frac{1}{2}(\Delta \phi)^2\right\} \, dx,
\end{align}
where $\phi$ is the density field, and $\epsilon$ $< 1$ is a constant. Following \cite{WWL:09:PFC}, we consider conserved gradient dynamics. The resulting equation is known as the phase field crystal (PFC) equation, 
\begin{align}\label{eq:pfc-pre}
\partial_t \phi = \nabla \cdot \left(M(\phi) \nabla \mu\right),
\end{align}
where $M(\phi)>0$ is a mobility, $\mu$ is the chemical potential defined as
\begin{align}\label{eq:chem-pot-pre}
\mu := \delta_\phi E = \phi^3 + (1-\epsilon) \phi + 2 \Delta \phi + \Delta^2 \phi,
\end{align} 
with $\delta_\phi E$ denoting the variational derivative of $E$ with respect to $\phi$, and where either the natural boundary conditions $\partial_n \phi = \partial_n \Delta \phi = \partial_n \mu = 0$ or periodic boundary conditions are assumed. In the remainder of the paper, we only consider the special case for which the mobility is held constant $M(\phi) \equiv \mathcal{M}$ and natural boundary conditions are assumed. Therefore, the system of equations which is considered herein is:
\begin{subequations}
\begin{align}
\partial_t \phi  &= \nabla \cdot \left(\mathcal{M} \nabla \mu\right), 
\label{eq:pfc}
\\
\mu &= \phi^3 + (1-\epsilon) \phi + 2 \Delta \phi + \Delta^2 \phi,
\label{eq:chem-pot}
\\
\partial_n \phi &= \partial_n \Delta \phi = \partial_n \mu = 0.
\label{eq:boundary-cond}
\end{align}
\end{subequations}

As stated above, many papers focusing on finite difference and spectral methods for the phase field crystal equation exist in the literature \cite{BRV:07:PFC,DFWWZ:18:PFCconv,HL:18:6oCH,HWWL:09:PFC,SLL:2016:PFC,VDBCC:2015:PFC,WWL:09:PFC}. We summarize a few of the most relevant references here. In \cite{WWL:09:PFC}, Wise, Wang, and Lowengrub present an energy stable and convergent finite difference scheme for the phase field crystal equation. They use a first order in time convex-splitting scheme for time discretization and show a detailed analysis of the energy stability resulting from the proposed time stepping approach. Additionally, they are able to show unique solvability and local-in-time error estimates which ensure the convergence of the scheme. In \cite{DFWWZ:18:PFCconv}, Dong et.~al.~present a convergence analysis and numerical implementation of the second order in time scheme originally considered in \cite{HWWL:09:PFC} where again spacial discretization is achieved via a finite difference method. 

Numerical schemes employing the finite element framework for the PFC model is lacking in the literature. One of the most challenging aspects of the PFC model with respect to a finite element discretization has to do with the fourth order term residing in the chemical potential \eqref{eq:chem-pot}. However, we refer the reader to two papers in the literature which present different methods for handling the fourth order term. In \cite{BRV:07:PFC}, Backofen, R\"{a}tz, and Voigt introduce a mixed formulation composed of three second order equations. For time discretization, they employ a first order in time finite element method which essentially follows a backward Euler scheme but where the nonlinear term in the chemical potential is linearized. A brief comparison of the time stepping scheme utilized in the Backofen paper to the time stepping scheme employed in \cite{WWL:09:PFC} is presented in the latter paper. We point out here that, due to their chosen mixed formulation, the resulting system to be solved is non-symmetric and indeterminant possibly resulting in large computational costs. In \cite{HL:18:6oCH}, Hoppe and Linsenmann introduce a C$^0$ interior penalty finite element method for a sixth order Cahn-Hilliard equation which models microemulsification processes and is closely related to the PFC model. However, they are only able to establish semi-optimal convergence results and are not able to show their numerical scheme preserves energy stability. Additionally, the work in \cite{HL:18:6oCH} tends to deviate from the framework established in the below mentioned literature on C$^0$ interior penalty methods. 

In this paper, we employ the C$^0$ interior penalty (C$^0$-IP) method for spacial discretization and the convex-splitting time discretization proposed in \cite{WWL:09:PFC} to the phase field crystal model \eqref{eq:pfc}-\eqref{eq:boundary-cond}.  The C$^0$-IP method is characterized by the use of C$^0$ Lagrange finite elements where the C$^1$ continuity requirement inherent to standard conforming finite element methods has been replaced with interior penalty techniques and was first introduced by G. Engel et al.~in \cite{EGHLMT:02:fourth} and revisited and analyzed by S.C.~Brenner and co-workers in \cite{brenner:11:frontiers,BGGS:12:C0CHboundary,BGS:09:posteriori,BN:11:C0elliptic,BS:05:C0IP,BS:06:multigrid,BSZZ:12:quadratic,BSZ:14:quadratic,BSZ:15:post,BW:05:schwarz,BWZ:04:poincare} also see \cite{FHP:15:convergence} and \cite{GG:2013:C0IPFK,GGN:2013:C0IP, GN:2011:sixthIP}. These works include \emph{a priori} and \emph{a posteriori} error analyses as well as multigrid and domain decomposition solvers for the plate bending problem. While C$^0$-IP methods have been around for almost two decades, its application to time dependent nonlinear problems with Cahn-Hilliard type boundary conditions is relatively recent and the authors are aware of only two such papers: the Hoppe and Linsenmann paper mentioned above and the paper \cite{GG:2013:C0IPFK} by Gudi and Gupta on a C$^0$ interior penalty method for the extended Fisher-Kolmogorov fourth order equation. 

In contrast to the work done by Hoppe and Linsenmann in \cite{HL:18:6oCH}, this paper presents a C$^0$-IP method for the PFC equation which is energy stable. Additionally, we are able to prove unique solvability and, most importantly, establish a framework for the convergence of C$^0$ interior penalty methods which follows in the framework of most of the existing literature regarding these methods. In contrast to the work completed by Gudi and Gupta in \cite{GG:2013:C0IPFK}, we are able to establish convergence results for the sixth order problem and in the case in which solutions to the weak form of the PDE are not elements of the finite element space. Furthermore, we present two numerical experiments which demonstrate the effectiveness of our method. Finally, we note that the convex-splitting time stepping strategy was chosen due to the nice properties regarding energy stability and unique solvability for $H^{-1}$ gradient flow problems. However, we emphasize that the novelty of this paper lies in the fact that this is the first C$^0$ interior penalty finite element method developed for the phase field crystal equation and the error analysis presented develops a detailed methodology for analyzing time dependent problems utilizing the C$^0$ interior penalty method. It is likely that a different time-stepping strategy would lead to higher-order convergence with respect to time (see for example \cite{calo2020splitting, ZMQ:13:adaptive}), and different strategies will be considered in future works.

The development of the C$^0$-IP method relies on a weak formulation of \eqref{eq:pfc}-\eqref{eq:boundary-cond}. To this end, we introduce the function space $Z:= \{z \in H^2(\Omega) | {\bf n} \cdot \nabla z = 0 \text{ on } \partial \Omega\}$ and remark that we use the standard Sobolev space and norm notation throughout the paper. In particular, we let $\norm{\cdot}{L^2}$ denote the standard $L^2$ norm over the region $\Omega$ but specify the notation $\norm{\cdot}{L^2(S)}$ as the $L^2$ norm over a general region $S \subset \mathbb{R}^2$ which is not $\Omega$.   A weak formulation of \eqref{eq:pfc}-\eqref{eq:boundary-cond} may then be written as follows \cite{PS:13:CHweak}: find $(\phi, \mu)$ such that 
\begin{subequations}
\begin{align}
\phi &\in L^{\infty}(0,T; Z) \cap L^2(0,T;H^3(\Omega)),
\\
\partial_t \phi &\in L^2(0,T; H_N^{-1}(\Omega)),
\\
\mu &\in L^2(0,T;H^1(\Omega)),
\end{align}
\end{subequations}
and there hold for almost all $t \in (0,T)$
\begin{subequations}\label{eq:pfc-weak}
\begin{align}
\langle \partial_t \phi, \nu \rangle + \iprd{\mathcal{M} \nabla \mu}{\nabla \nu} &= 0 &\forall \, \nu &\in H^1(\Omega)
\label{eq:pfc-weak-a}
\\
\iprd{\left(\phi\right)^3 + (1-\epsilon)\phi}{\psi} - 2 \iprd{ \nabla \phi}{ \nabla \psi } + \aiprd{ \phi }{\psi} - \iprd{ \mu}{ \psi}&= 0 &\forall \, \psi &\in Z
\end{align}
\end{subequations}
with the compatible initial data 
\begin{align}
\phi(0) = \phi_0 \in H^4(\Omega) \text{ such that } {\bf n} \cdot \nabla \phi_0 = 0 \text{ and } {\bf n} \cdot \nabla \Delta \phi_0 = 0,
\end{align}
and where $(u,v)$ is the $L^2(\Omega)$ inner product of $u$ and $v$ and $\aiprd{u}{v}:= \left(\nabla^2 u : \nabla^2 v \right)$ is the inner product of the Hessian matrices of $u$ and $v$. Additionally, we use the notations $H_N^{-1}(\Omega)$ to indicate the dual space of $H^1(\Omega)$ and $\langle  \, \cdot \, , \, \cdot \, \rangle$ to indicate a duality pairing. Throughout the paper, we use the notation $\Phi(t) := \Phi(\, \cdot \, , t)\in X$, which views a spatiotemporal function as a map from the time interval $[0,T]$ into an appropriate Banach space, $X$.  The system \eqref{eq:pfc-weak} is mass conservative:  for almost every $t\in[0,T]$, $\iprd{\phi(t)-\phi_0}{1} = 0$. This property may be observed by setting $\nu = 1$ in \eqref{eq:pfc-weak-a}.  


The paper proceeds as follows. Section \ref{sec:fem} develops the fully discrete C$^0$ interior penalty finite element method for the phase field crystal model. Section \ref{sec:unique-solvability-stability} establishes unconditional unique solvability and unconditional stability. Section \ref{sec:error-estimates} presents the error analysis. Section \ref{sec:numerical} demonstrates the effectiveness of our method through two numerical experiments and we conclude in Section \ref{sec:conclusion}.

\section{A C$^0$ Interior Penalty Finite Element Method}\label{sec:fem}

 \par\noindent
In this section, we develop a fully discrete C$^0$-IP method for the phase field crystal equation \eqref{eq:pfc}--\eqref{eq:boundary-cond}. Throughout the remainder of the paper we consider only the case with $\mathcal{M} \equiv 1$ but note that the results will hold for any $\mathcal{M} > 0$. Let $\TT_h$ be a geometrically conforming, locally quasi-uniform simplicial triangulation of $\Omega$. We introduce the following notation:
\begin{itemize}\itemsep0ex
\item $h_K = $ diameter of triangle $K$ ($h = \max_{K \in \TT_h} h_K$),
\item $v_K = $ restriction of the function $v$ to the triangle $K$,
\item $|K| = $ area of the triangle $K$,
\item $\Eh = $ the set of the edges of the triangles in $\TT_h$,
\item $e = $ the edge of a triangle,
\item $|e| = $ the length of the edge,
\item $V_h := \{ v \in C(\overline{\Omega}) | v_K  \in P_1(K) \forall K \in \TT_h \}$ the standard finite element space associated with $\TT_h$ of degree 1,
\item $Z_h := \{ v \in C(\overline{\Omega}) | v_K  \in P_2(K) \forall K \in \TT_h \}$ the standard Lagrange finite element spaces associated with $\TT_h$ of degree 2.
\end{itemize}

Let $M$ be a positive integer such that $ t_m = t_{m-1} + \tau $ for $\ 1 \le m \le M$ where $t_0 = 0 , \ t_M = t_F$ with $\tau = \nicefrac{t_F}{M}$.  A fully discrete C$^0$ interior penalty method for \eqref{eq:pfc-weak} is: given $\phih^{m-1} \in Z_h$, find $\phih^{m}, \muh^m \in  Z_h \times V_h$ such that 
\begin{subequations} \label{eq:C0IP-scheme}
\begin{align}
\label{eq:C0IP-scheme-a}
\iprd{\dtau \phih^m}{\nu_h} + \iprd{\nabla \muh^m}{\nabla \nu_h} &= 0, \forall \, \nu_h \in V_h
\\
\label{eq:C0IP-scheme-b}
 \aIPh{\phih^m}{\psi_h} +\iprd{\left(\phih^m\right)^3 + (1-\epsilon)\phih^m}{\psi_h} - 2 \iprd{ \nabla \phih^{m-1}}{ \nabla \psi_h }- \iprd{ \muh^m}{ \psi_h}&= 0, \forall \, \psi_h \in Z_h,
\end{align}
\end{subequations}
with initial data taken to be $\phih^0 := P_h \phi_0 = P_h \phi(0)$ where $P_h: Z \rightarrow Z_h$ is a Ritz projection operator (reminiscent of the projection defined in~\cite[p.~887]{elliott1989nonconforming}) such that
\begin{align}\label{eq:H2-Ritz-projection}
\aIPh{P_h \phi - \phi}{\xi} + (1-\epsilon) \iprd{P_h \phi - \phi}{\xi}  = 0 \quad \forall \, \xi \in Z_h, \quad \iprd{P_h \phi - \phi}{1} = 0,
\end{align}
and where $\displaystyle \dtau \phih^{m} := \frac{\phih^{m} - \phih^{m-1}}{\tau}$. The bilinear form $\aIPh{\cdot}{\cdot}$ is defined by
\begin{align}
\label{eq:aIPh-def}
\aIPh{w}{v} :=& \, \sum_{K \in \TT_h}  \int_K \left(\nabla^2 w : \nabla^2 v\right) \, dx + \sum_{e \in \Eh} \int_e \dgal[\Bigg]{\frac{\partial^2 w}{\partial n_e^2} } \left\llbracket \frac{\partial v}{\partial n_e} \right\rrbracket dS 
\nonumber
\\
&+ \sum_{e \in \Eh} \int_e \dgal[\Bigg]{\frac{\partial^2 v}{\partial n_e^2} } \left\llbracket \frac{\partial w}{\partial n_e} \right\rrbracket dS + \alpha \sum_{e \in \Eh} \frac{1}{|e|} \int_e  \left\llbracket \frac{\partial w}{\partial n_e} \right\rrbracket  \left\llbracket \frac{\partial v}{\partial n_e} \right\rrbracket dS,
\end{align}
with $\alpha \ge 1$ known as a penalty parameter. The jumps and averages that appear in \eqref{eq:aIPh-def} are defined as follows. For an interior edge $e$ shared by two triangles $K_\pm$ where $n_e$ points from $K_-$ to $K_+$, we define on the edge $e$
\begin{align}
\label{eq:jumps-ave-interior}
\left\llbracket \frac{\partial v}{\partial n_e} \right\rrbracket = n_e \cdot \left(\nabla v_+ - \nabla v_-\right) \quad \text{and} \quad \dgal[\Bigg]{\frac{\partial^2 v}{\partial n_e^2} } = \frac{1}{2}\left( \frac{\partial^2 v_-}{\partial n_e^2} +  \frac{\partial^2 v_+}{\partial n_e^2} \right),
\end{align}
where $\displaystyle \frac{\partial^2 u}{\partial n_e^2} = n_e \cdot \left(\nabla^2 u\right) n_e$ and where $v_\pm = v |_{K_\pm}$. For a boundary edge $e$ which is an edge of the triangle $K \in \TT_h$, we take $n_e$ to be the unit normal pointing towards the outside of $\Omega$ and define on the edge $e$
\begin{align}\label{eq:jumps-ave-boundary}
\left\llbracket \frac{\partial v}{\partial n_e} \right\rrbracket = - n_e \cdot \nabla v_K \quad \text{and} \quad \dgal[\Bigg]{\frac{\partial^2 v}{\partial n_e^2} } = n_e \cdot \left( \nabla^2 v \right) n_e.
\end{align}

\begin{rmk}
Note that the definitions \eqref{eq:jumps-ave-interior} and \eqref{eq:jumps-ave-boundary} are independent of the choice of $K_\pm$, or equivalently, independent of the choice of $n_e$ \cite{brenner:11:frontiers}.
\end{rmk}

\section{Unique Solvability and Stability}\label{sec:unique-solvability-stability}

In this section, we show that the C$^0$-IP method for the PFC equation outlined in the previous section admits a unique solution and that the system follows an energy law similar to \eqref{eq:pfc-energy}. In order to show the existence of a unique solution and unconditional energy stability, we will need the following definitions and lemma. First, we define the following mesh dependent norm
\begin{align}\label{def:coip-norm}
\norm{v_h}{2,h}^2:=\sum_{K\in \TT_h} |v_h|_{H^2(K)}^2 + \sum_{e\in \Eh} \frac{\alpha}{|e|} \norm{\left\llbracket \frac{\partial v_h}{\partial n_e} \right\rrbracket}{L^2(e)}^2.
\end{align}
The next lemma guarantees the boundedness of $\aIPh{\cdot}{\cdot}$.

\begin{lemma}[Boundedness of $\aIPh{\cdot}{\cdot}$]\label{lem:aIPh-boundedness} There exists positive constants $C_{cont}$ and $C_{coer}$ such that for choices of the penalty parameter $\alpha$ large enough we have
\begin{align}
\aIPh{w_h}{v_h} &\le C_{cont} \norm{w_h}{2,h} \norm{v_h}{2,h} \quad \forall \, w_h, v_h \, \in Z_h,  \label{eq:cty}
\\
C_{coer}\norm{w_h}{2,h}^2 &\le \aIPh{w_h}{w_h} \quad \forall \, w_h \in Z_h, \label{eq:coer}
\end{align}
where the constants $C_{cont}$ and $C_{coer}$ depend only on the shape regularity of $\TT_h$.
\end{lemma}
\begin{proof}
The proof of the Lemma may be found in \cite{brenner:11:frontiers}.
\end{proof}

Additionally, we define the spaces $L_0^2(\Omega) := \{v \in L^2(\Omega) | \iprd{v}{1} = 0 \}, \mathring{H}^1(\Omega) := H^1(\Omega) \cap L_0^2(\Omega), \mathring{H}_N^{-1}(\Omega):=\{v \in H_N^{-1}(\Omega) | \langle v, 1\rangle = 0\}, \Voh := V_h\cap L_0^2(\Omega)$, and $\Zoh := Z_h \cap L_0^2(\Omega)$. The operator $\mathsf{T}: \mathring{H}^{-1}_N(\Omega) \rightarrow \mathring{H}^1(\Omega) $ is often referred to as the `inverse Laplacian' and is defined via the following variational problem: given $\zeta\in \mathring{H}_N^{-1}(\Omega)$, find $\mathsf{T}\zeta\in \mathring{H}^1(\Omega)$ such that
	\begin{equation}
	\label{def:inverse-laplacian}
\iprd{\nabla\mathsf{T} \zeta}{\nabla\chi} = \langle \zeta, \chi\rangle \qquad \forall \,\chi\in \mathring{H}^1(\Omega).
	\end{equation} 
The well posedness of the operator $\mathsf{T}$ is well known, see for example \cite{DFW:15:CHDS}, and an induced negative norm may be defined such that $\norm{v}{H_N^{-1}} = \iprd{\nabla \mathsf{T} v}{\nabla \mathsf{T} v}^{\nicefrac{1}{2}}  = \langle v, \mathsf{T} v \rangle^{\nicefrac{1}{2}} = \langle \mathsf{T} v, v \rangle^{\nicefrac{1}{2}}$. We furthermore define a discrete analog of the inverse Laplacian, $\mathsf{T}_h: \Zoh \rightarrow \Zoh$, via the variational problem: given $\zeta\in \Zoh$, find $\mathsf{T}_h \zeta\in \Zoh$ such that
	\begin{equation}
	\label{def:discrete-inverse-laplacian}
\iprd{\nabla\mathsf{T}_h \zeta_h}{\nabla\chi_h} = \iprd{\zeta_h}{\chi_h} \qquad \forall \, \chi_h\in \Zoh.
	\end{equation}
Again, the well posedness of the operator $\mathsf{T}_h$ is well known and an induced discrete negative norm on $\Zoh$ is defined as $\norm{v_h}{-1,h} = \iprd{\nabla \mathsf{T}_h v_h}{\nabla \mathsf{T}_h v_h}^{\nicefrac{1}{2}}  = \iprd{ v_h}{ \mathsf{T}_h v_h }^{\nicefrac{1}{2}} = \iprd{ \mathsf{T}_h v_h}{ v_h}^{\nicefrac{1}{2}}$.

\subsection{Unconditional Unique Solvability}{\label{sec:solvability}
In this section, we demonstrate that the scheme \eqref{eq:C0IP-scheme-a}--\eqref{eq:C0IP-scheme-b} is uniquely solvable for any mesh parameters $\tau$ and $h$ and for any of the model parameters such that $\epsilon <1$. 

\begin{lemma}
The scheme~\eqref{eq:C0IP-scheme} satisfies the discrete conservation property $\iprd{\phih^m}{1} = \iprd{\phih^0}{1} = \iprd{P_h \phi_0}{1} = \iprd{\phi_0}{1}$ for any $1 \le m \le M$. 
\end{lemma}

\begin{proof}
The result can be clearly observed by setting $\nu_h \equiv 1$ in \eqref{eq:C0IP-scheme-a}. 
\end{proof}

\begin{rmk} 
The quantity $ \frac{1}{|\Omega|} \iprd{\phi_0}{1} $ is referred to as the average of $\phi_0$ over $\Omega$ and is denoted by $\avephio$. Due to the discrete conservation property, it follows that $\iprd{\phih^m}{1} = \iprd{\phih^0}{1} = |\Omega| \, \avephio$.
\end{rmk}

\begin{lemma}\label{lem:grad-split}
Suppose $\Omega$ is a bounded polygonal domain. For all $w_h \in Z_h, v \in H^1(\Omega)$, $\gamma > 0$, and $\alpha$ large enough, 
\begin{align}
    |\iprd{\nabla w_h}{\nabla v}| \le \sqrt{(1+\gamma)} \norm{w_h}{2,h} \norm{v}{L^2}.
\end{align}
\end{lemma}

\begin{proof}
We begin with the integration by part formula:
\begin{align*}
\int_{K} \nabla w_h \cdot \nabla v \, dx = \int_{\partial K} \dfrac{\partial w_h}{\partial n} v \, ds - \int_{K} \Delta w_h v \, dx.
\end{align*}
Summing over all triangles in $\TT_h$ and noting that the sum involving the integral over the boundary of each triangle can be written as a sum over the edges in $\Eh$, we have
\begin{align*}
\sum_{K \in \TT_h} \int_{K} \nabla w_h \cdot \nabla v \, dx = \sum_{e \in \Eh} \int_{e} \left\llbracket \frac{\partial w_h}{\partial n_e} \right\rrbracket v \, ds - \sum_{K \in \TT_h} \int_{K} \Delta w_h v \, dx.
\end{align*}
Using the Young's inequality and a standard trace inequality \cite{CFX:17:WG,Riviere:08:DG} we have
\begin{align*}
\left(\sum_{K \in \TT_h} \int_{K} \nabla w_h \cdot \nabla v \, dx\right)^2 \le&\,  \left(1 + \frac{1}{\gamma}\right) \left(\sum_{e \in \Eh} |e|^{-1} \norm{\left\llbracket \frac{\partial w_h}{\partial n_e} \right\rrbracket}{L^2(e)}^2\right) \left(\sum_{e \in \Eh} |e| \norm{v}{L^2(e)}^2 \right)
\\
&+(1+\gamma) \left(\sum_{K \in \TT_h} |w_h|_{H^2(K)}^2 \right)\left(\sum_{K \in \TT_h} \norm{v}{L^2(K)}^2 \right)
\\
\le&\,  (1+\gamma) \norm{w_h}{2,h}^2 \norm{v}{L^2}^2
\end{align*}
for $\alpha$ large enough.
\end{proof}
}

	\begin{lemma}
	\label{thm:pfc-uniqueness}
 Let $\varphih^{m-1}\in \Zoh$ be given.  For all $\varphih\in \Zoh$, define the nonlinear functional 
	\begin{align}
G_h(\varphih) :=& \, \frac{\tau}{2}\norm{\frac{\varphih-\varphih^{m-1}}{\tau}}{-1,h}^2 + \frac{1}{2}\aIPh{\varphih}{\varphih} +\frac{1}{4}\norm{\varphih + \avephio}{L^4}^4  
\nonumber
\\
&+\frac{1-\epsilon}{2}\norm{\varphih+\avephio}{L^2}^2 -2 \iprd{\nabla \varphih^{m-1}}{\nabla \varphih} .
	\end{align}
The functional $G_h$ is strictly convex and coercive on the linear subspace $\Zoh$. Consequently, $G_h$ has a unique minimizer, call it  $\varphih^m\in \Zoh$.  Moreover, $\varphih^m\in \Zoh$ is the unique minimizer of $G_h$ if and only if it is the unique solution to
	\begin{equation}
\aIPh{\varphih^m}{\psi_h} + \iprd{\left(\varphih^m + \avephio\right)^3}{\psi_h} + (1-\epsilon)\iprd{\varphih^m + \avephio}{\psi_h} - \iprd{\mu_{h,\star}^m}{\psi_h} = 2 \iprd{ \nabla \varphih^{m-1}}{ \nabla \psi_h }  
	\label{eq:pfc-nonlinear-1}
	\end{equation}
for all $\psi_h \in \Zoh$, where $\mu_{h,\star}^m \in \Voh$ is the unique solution to
	\begin{align}
\iprd{\nabla \mu_{h,\star}^m}{\nabla \nu_h}  = -\iprd{\frac{\varphih^m-\varphih^{m-1}}{\tau}}{\nu_h} &  \qquad \forall \,  \nu_h \in\Voh.
	\label{eq:pfc-nonlinear-2}
	\end{align}
	\end{lemma}
	
The proof of Lemma \ref{thm:pfc-uniqueness} follows from a convexity argument similar to the proof of existence and uniqueness for the solution to the finite difference method developed by Wise et.~al.~in \cite{WWL:09:PFC} and for a convex-splitting finite element method for the Cahn-Hilliard-Darcy-Stokes system found in \cite{DFW:15:CHDS}. We have included the details in Appendix \ref{app:uniqueness} for the interested reader.

The next theorem demonstrates the unconditional unique solvability of our scheme.

	\begin{theorem}
	\label{thm:pfc-existence-uniqueness}
The scheme \eqref{eq:pfc-nonlinear-1} -- \eqref{eq:pfc-nonlinear-2} is uniquely solvable for any mesh parameters $\tau$ and $h$ and for any $\epsilon < 1$. Furthermore, the scheme \eqref{eq:pfc-nonlinear-1} -- \eqref{eq:pfc-nonlinear-2} is equivalent to the scheme \eqref{eq:C0IP-scheme-a} -- \eqref{eq:C0IP-scheme-b}. Thus, the scheme \eqref{eq:C0IP-scheme-a} -- \eqref{eq:C0IP-scheme-b} is uniquely solvable for any mesh parameters $\tau$ and $h$ and for any $\epsilon < 1$.
	\end{theorem}
	
	\begin{proof}
Suppose $\iprd{\varphih^{m-1}}{1}=0$. It is clear that a necessary condition for solvability of \eqref{eq:pfc-nonlinear-1} -- \eqref{eq:pfc-nonlinear-2} is that 
	\begin{equation}
\left(\varphih^m,1\right) = \bigl(\varphih^{m-1},1\bigr) = 0,
	\end{equation}
as can be found by taking $\nu_h \equiv 1$ in \eqref{eq:pfc-nonlinear-2}. Now, let $ \varphih^m,\mu_{h,\star}^m  \in \Zoh \times \Voh$ be a solution of \eqref{eq:pfc-nonlinear-1} -- \eqref{eq:pfc-nonlinear-2}.  Set
	\begin{equation}
\overline{\mu_h^m}:=\frac{1}{|\Omega|}\iprd{(\varphih^m+\avephio)^3 + (1-\epsilon)\left(\varphih^m+\avephio\right)}{1} =\frac{1}{|\Omega|}\iprd{(\varphih^m+\avephio)^3 }{1} + (1-\epsilon)\avephio ,
	\label{eq:chds-chem-pot-average}
	\end{equation}
and define $\muh^m:=\mu_{h,\star}^m+\overline{\mu_h^m}$.  There is a one-to-one correspondence of the respective solution sets: $\varphih^m,\mu_{h,\star}^m \in \Zoh \times \Voh$ is a solution to \eqref{eq:pfc-nonlinear-1} -- \eqref{eq:pfc-nonlinear-2} if and only if $ \phih^m,\muh^m \in Z_h\times V_h$ is a solution to \eqref{eq:C0IP-scheme-a} -- \eqref{eq:C0IP-scheme-b}, where
	\begin{equation}
\phih^m = \varphih^m+\avephio,\quad \muh^m = \mu_{h,\star}^m+\overline{\mu_h^m}.
	\end{equation}
But \eqref{eq:pfc-nonlinear-1} -- \eqref{eq:pfc-nonlinear-2} admits a unique solution, which proves that \eqref{eq:C0IP-scheme-a} -- \eqref{eq:C0IP-scheme-b} is uniquely solvable.
	\end{proof}

\subsection{Unconditional Stability}\label{sec:stability}

Energy stability follows as a direct result of the convex decomposition represented in the scheme. First, we define a discrete energy closely related to \eqref{eq:pfc-energy},
\begin{align}
\label{eq:discrete-energy}
F(\phi) := \frac{1}{4} \norm{\phi}{L^4}^4 + \frac{1-\epsilon}{2}\norm{\phi}{L^2}^2 - \norm{\nabla \phi}{L^2}^2 + \frac{1}{2} \aIPh{\phi}{\phi}.
\end{align}

	\begin{lemma}
	\label{lem:discrete-energy-law}
Let $(\phih^m, \muh^m) \in Z_h \times V_h$ be a solution of \eqref{eq:C0IP-scheme-a}--\eqref{eq:C0IP-scheme-b}. Then the following energy law holds for any $h,\tau>0$:
	\begin{align}
F\left(\phih^\ell\right) &+\tau \sum_{m=1}^\ell \norm{\nabla \muh^m}{L^2}^2 + \tau^2 \sum_{m=1}^\ell \Biggl\{ \,  \frac{(1-\epsilon)}{2} \norm{\dtau  \phih^m}{L^2}^2  + \norm{\nabla \dtau  \phih^m}{L^2}^2 
	\nonumber
	\\
&\quad + \frac{1}{4}\norm{\dtau (\phih^m)^2}{L^2}^2 + \frac{1}{2}\norm{\phih^m\dtau \phih^m}{L^2}^2 +\frac{1}{2}\aIPh{\dtau \phih^m}{\dtau \phih^m} \, \Biggr\} = F\left(\phih^0\right),
	\label{eq:ConvSplitEnLaw}
	\end{align}
for all $1\leq \ell \leq M$.
	\end{lemma}

\begin{proof}
Setting $\nu_h = \muh^m$ in \eqref{eq:C0IP-scheme-a} and $\psi_h =  \dtau \phih^m$ in \eqref{eq:C0IP-scheme-b}, we have
\begin{align*}
\iprd{\dtau \phih^m}{\muh^m} + \iprd{\nabla \muh^m}{\nabla \muh^m} &= 0 ,
\\
\iprd{\left(\phih^m\right)^3 + (1-\epsilon)\phih^m}{\dtau \phih^m} + \aIPh{\phih^m}{\dtau \phih^m} -2 \iprd{ \nabla \phih^{m-1}}{ \nabla \dtau \phih^m } - \iprd{ \muh^m}{ \dtau \phih^m}&= 0.
\end{align*}
Note that $(\phih^m)^3 = \nicefrac{1}{2} \left[(\phih^m)^2 \cdot (\phih^m - \phih^{m-1}) +  (\phih^m)^2 \cdot (\phih^m + \phih^{m-1}) \right]$.  Adding the two equations together and using the polarization identity, we obtain
\begin{align*}
\norm{\nabla \muh^m}{L^2}^2 + \frac{1}{4\tau} \left(\norm{\phih^m}{L^4}^4 - \norm{\phih^{m-1}}{L^4}^4 \right) + \frac{\tau}{4} \norm{\dtau (\phih^m)^2}{L^2}^2+ \frac{\tau}{2}\norm{\phih^m\dtau \phih^m}{L^2}^2 
\\
+ \frac{(1-\epsilon)}{2\tau} \left(\norm{\phih^m}{L^2}^2 - \norm{\phih^{m-1}}{L^2}^2 \right) + \frac{(1-\epsilon)\tau}{2} \norm{\dtau \phih^m}{L^2}^2 
\\
+ \frac{1}{2\tau} \aIPh{\phih^m}{\phih^m} - \frac{1}{2\tau} \aIPh{\phih^{m-1}}{\phih^{m-1}} + \frac{\tau}{2} \aIPh{\dtau \phih^m}{\dtau \phih^m} 
\\
- \frac{1}{\tau} \left(\norm{\nabla \phih^m}{L^2}^2 - \norm{\nabla \phih^{m-1}}{L^2}^2\right) + \tau \norm{\nabla \dtau \phih^m}{L^2}^2 = 0.
\end{align*}
Applying $\tau \sum\limits_{m=1}^\ell$ gives the desired result.
\end{proof}

The discrete energy law implies the following uniform {\em a priori} estimates for $\phih^m$ and $\muh^m$. 

\begin{lemma}\label{lem:stability}
Let $(\phih^m, \muh^m) \in Z_h\times V_h$ be the unique solution of \eqref{eq:C0IP-scheme-a}--\eqref{eq:C0IP-scheme-b}. Suppose that $F(\phih^0) \le C$ independent of $h$ and that $\epsilon < 1 + \frac{(1-\gamma)C_{coer} - 1}{C_{coer}}$. Then the following estimates hold for any $\tau, h >0$:
\begin{align}
\max_{0 \le m \le M} \left[  \norm{\phih^m}{L^4}^2 + \norm{\phih^m}{L^2}^2 +  \norm{\phih^m}{H^1}^2 +\norm{\phih^m}{2,h}^2\right] &\le C,
\label{eq:2,h-stability}
\\
\tau \sum_{m=1}^\ell \norm{\nabla \muh^m}{L^2}^2 & \le C,
\\
\tau^2 \sum_{m=1}^\ell \Biggl\{ \,  \norm{\nabla \dtau  \phih^m}{L^2}^2 + \norm{\phih^m \dtau \phih^m}{L^2}^2 + \norm{\dtau \phih^m}{2,h}^2 \, \Biggr\} & \le C,
\end{align}
for some constant $C$ that is independent of $h, \tau,$ and $T$.
\end{lemma}

\begin{proof}
First, note that since $(a^2-1)^2\ge0$, then we have
\begin{align}\label{eq:L4-bound}
\frac{1}{4}\norm{u}{L^4}^4 \ge \frac{1}{2}\norm{u}{L^2}^2 - \frac{|\Omega|}{4}
\end{align}
for any $u \in L^4(\Omega)\cap L^2(\Omega)$. Thus, as a result of Lemma \ref{lem:discrete-energy-law} and equation \eqref{eq:coer}, we have for any $0 \le m \le M$
\begin{align*}
 &\frac{1}{2} \norm{\phih^m}{L^2}^2 - \frac{|\Omega|}{4} + \frac{1-\epsilon}{2}\norm{\phih^m}{L^2}^2 - \norm{\nabla \phih^m}{L^2}^2 + \frac{C_{coer}}{2} \norm{\phih^m}{2,h}^2 
 \\
 &\quad\le  \frac{1}{4} \norm{\phih^m}{L^4}^4 + \frac{1-\epsilon}{2}\norm{\phih^m}{L^2}^2 - \norm{\nabla \phih^m}{L^2}^2 + \frac{1}{2} \aIPh{\phih^m}{\phih^m} \le F(\phih^0) \le C.
\end{align*}
Rearranging a few terms and invoking Lemma \ref{lem:grad-split} and Young's inequality, we have
\begin{align*}
\frac{2-\epsilon}{2}\norm{\phih^m}{L^2}^2 + \frac{C_{coer}}{2} \norm{\phih^m}{2,h}^2 &\le C + \norm{\nabla \phih^m}{L^2}^2
\\
&\le C + \sqrt{(1+\gamma)} \norm{\phih^m}{L^2} \norm{\phih^m}{2,h} 
\\
&\le C + \frac{(2-\epsilon)}{2} \norm{\phih^m}{L^2}^2 + \frac{(1+\gamma)}{2(2-\epsilon)} \norm{\phih^m}{2,h}^2 .
\end{align*}
The last term in estimate \eqref{eq:2,h-stability} follows as long as
\begin{align*}
C_{coer} > \frac{(1+\gamma)}{(2-\epsilon)} \quad \text{or} \quad \epsilon < 1 + \frac{(1-\gamma)C_{coer} - 1}{C_{coer}}.
\end{align*}
The remainder of the proof follows as a result of Lemma \ref{lem:discrete-energy-law}.  
\end{proof}

\begin{rmk}
Following \cite{brenner:11:frontiers}, we note that $C_{coer}$ can be chosen to be close to 1 as long as the penalty parameter $\alpha$ is large enough. In this case, $\gamma$ could also be chosen close to 0 and \eqref{eq:2,h-stability} will hold as long as $\epsilon < 1$. 
\end{rmk}


\section{Error Estimates}\label{sec:error-estimates}

In this section, we provide a rigorous convergence analysis for the semi-discrete method in the appropriate energy norms. We shall assume that the weak solutions have the additional regularities
\begin{align}
\phi &\in L^{\infty}\left(0,T;H^3(\Omega)\right) \cap L^2\left(0,T;H^3(\Omega)\right),
\nonumber
\\
\partial_t \phi &\in L^2\left(0,T;H^3(\Omega)\right) \cap  L^2(0,T; H_N^{-1}(\Omega)),
\nonumber
\\
\partial_{tt} \phi &\in L^2\left(0,T;L^2(\Omega)\right) ,
\nonumber
\\
\mu &\in L^2\left(0,T;H^2(\Omega)\right),
\nonumber
\\
\partial_t \mu &\in L^2\left(0,T;L^2(\Omega)\right).
	\label{eq:higher-regularities}
	\end{align}

The interior penalty method \eqref{eq:C0IP-scheme-a}--\eqref{eq:C0IP-scheme-b} is not well-defined for solutions to \eqref{eq:pfc-weak} since $Z_h \not\subset Z$. Therefore, we define $W_h \subset Z$ to be the Hsieh-Clough-Tocher micro finite element space associated with $\TT_h$ as in \cite{BGGS:12:C0CHboundary}. We furthermore define the linear map $E_h: Z_h \rightarrow W_h \cap Z$ as in \cite{BGGS:12:C0CHboundary} which allows us to consider the following problem: find $(\phi, \mu) \in Z \times H^1(\Omega)$ such that
\begin{subequations}\label{eq:cont-correction}
\begin{align}
\iprd{\partial_t \phi}{\nuh} + \iprd{\nabla \mu}{\nabla \nuh} = 0, & \quad \forall \, \nuh \in V_h,
\label{eq:cont-correction-a}
\\
\aIPh{\phi}{\psi_h} +\iprd{\left(\phi\right)^3 + (1-\epsilon)\phi}{\psi_h} - 2 \iprd{ \nabla \phi}{ \nabla \psi_h }- \iprd{ \mu}{ \psi_h} &
\nonumber
\\
= \aIPh{\phi}{\psi_h- E_h \psih} +\iprd{\left(\phi\right)^3 + (1-\epsilon)\phi}{\psi_h- E_h \psih} &
\nonumber
\\
- 2 \iprd{ \nabla \phi}{ \nabla \psi_h - \nabla E_h \psih} - \iprd{ \mu}{ \psi_h- E_h \psih}&, \quad \forall \, \psi_h \in Z_h.
\label{eq:cont-correction-b}
\end{align}
\end{subequations}
Note that solutions of \eqref{eq:cont-correction} are consistent with solutions of \eqref{eq:pfc-weak} since $\aIPh{\phi}{E_h \psi} = \aiprd{\phi}{E_h \psi}$ for all $\psi \in Z_h$. 

\begin{rmk}
One of the primary challenges in the error analysis to follow arises due to insufficient global regularity possessed by solutions to \eqref{eq:pfc-weak} in the space $Z_h$. To remedy this, we rely on considering the Hsieh-Clough-Tocher micro finite element space associated with $\TT_h$ with the help of the enriching operator $E_h: Z_h \rightarrow W_h \cap Z$ as in \cite{BGGS:12:C0CHboundary}. This new weak formulation is well defined on the finite element spaces and additionally illustrates the error which is encountered by utilizing a non-conforming method such as the C$^0$ interior penalty method. 
\end{rmk}

Additionally, we introduce the following notation:
\begin{align*}
\err^{\phi,m} = \errP^{\phi,m} + \errh^{\phi,m}, \quad \errP^{\phi,m} := \phi^m - P_h \phi^m, \quad \errh^{\phi,m} := P_h \phi^m - \phih^m,
\\
\err^{\mu,m} = \errR^{\mu,m} + \errh^{\mu,m}, \quad \errR^{\mu,m} := \mu^m - R_h \phi^m, \quad \errh^{\mu,m} := R_h \phi^m - \phih^m,
\end{align*}
where $\phi^m := \phi(t_m)$ and $R_h: H^1(\Omega) \rightarrow V_h$ is a Ritz projection operator such that
\begin{align}\label{eq:H1-Ritz-projection}
\iprd{\nabla \left(R_h  \mu - \mu \right)}{\nabla \xi} = 0 \quad \forall \, \xi \in V_h, \quad \iprd{R_h \mu - \mu}{1} = 0.
\end{align} 
Using this notation and subtracting \eqref{eq:C0IP-scheme} from \eqref{eq:cont-correction}, we have for all $\nuh \in V_h$ and $\psih \in Z_h$
\begin{align*}
&\iprd{\dtau \err^{\phi,m}}{\nuh} + \iprd{\nabla \err^{\mu,m}}{\nabla \nuh} = \iprd{\dtau \phi^{m} - \partial_t \phi^{m}}{\nuh} ,
\\
\\
&\aIPh{\err^{\phi,m}}{\psi_h} + \iprd{\left(\phi^m\right)^3 - \left(\phih^m\right)^3}{\psih} +\iprd{(1-\epsilon)\err^{\phi,m}}{\psi_h} - 2 \iprd{ \nabla \err^{\phi,m-1}}{ \nabla \psi_h }- \iprd{ \err^{\mu,m}}{ \psi_h} 
\\
&\quad= - 2 \iprd{ \nabla \phi^{m-1} - \nabla \phi^m}{ \nabla \psi_h} + \aIPh{\phi^m}{\psi_h- E_h \psih} +\iprd{\left(\phi^m\right)^3 + (1-\epsilon)\phi^m}{\psi_h- E_h \psih} 
\\
&\qquad- 2 \iprd{ \nabla \phi^{m}}{ \nabla \psi_h - \nabla E_h \psih} - \iprd{ \mu^m}{ \psi_h- E_h \psih}.
\end{align*}
Invoking the properties of the Ritz projection operators, we have for all $\nuh \in V_h$ and all $\psih \in Z_h$
\begin{align}
&\iprd{\dtau \errh^{\phi,m}}{\nuh} + \iprd{\nabla \errh^{\mu,m}}{\nabla \nuh} = \iprd{\dtau \phi^{m} - \partial_t \phi^{m}}{\nuh} - \iprd{\dtau \errP^{\phi,m}}{\nuh},
\label{eq:error-eq-a}
\\
\nonumber
\\
&\aIPh{\errh^{\phi,m}}{\psi_h}  +\iprd{(1-\epsilon)\errh^{\phi,m}}{\psi_h} - 2 \iprd{ \nabla \errh^{\phi,m-1}}{ \nabla \psi_h } - \iprd{ \errh^{\mu,m}}{ \psi_h} 
\nonumber
\\
&\quad=  2 \iprd{ \nabla \errP^{\phi,m-1}}{ \nabla \psi_h } +\iprd{ \errR^{\mu,m}}{ \psi_h} - \iprd{\left(\phi^m\right)^3 - \left(\phi^m\right)^3}{\psih} - 2 \iprd{ \nabla \phi^{m-1} - \nabla \phi^m}{ \nabla \psi_h} 
\nonumber
\\
&\qquad+ \aIPh{\phi^m}{\psi_h- E_h \psih} + \iprd{\left(\phi^m\right)^3 + (1-\epsilon)\phi^m}{\psi_h- E_h \psih} 
\nonumber
\\
&\quad \qquad- 2 \iprd{ \nabla \phi^{m}}{ \nabla \psi_h - \nabla E_h \psih} - \iprd{ \mu^m}{ \psi_h- E_h \psih}.
\label{eq:error-eq-b}
\end{align}
Setting $\nuh = \errh^{\mu,m}$ and $\psih = \dtau \errh^{\phi,m}$ and adding and subtracting $4\iprd{\errh^{\phi,m}}{\dtau \errh^{\phi,m}}$, we arrive at the key error equation
\begin{align}\label{eq:main-error}
&\norm{\nabla \errh^{\mu,m}}{L^2}^2 + \aIPh{\errh^{\phi,m}}{\dtau \errh^{\phi,m}} + 4\iprd{\errh^{\phi,m}}{\dtau \errh^{\phi,m}}  + \iprd{(1-\epsilon)\errh^{\phi,m}}{\dtau \errh^{\phi,m}} 
\nonumber
\\
&\quad- 2 \iprd{ \nabla \errh^{\phi,m-1}}{ \nabla \dtau \errh^{\phi,m} } = \iprd{\dtau \phi^{m} - \partial_t \phi^{m}}{\errh^{\mu,m}} - \iprd{\dtau \errP^{\phi,m}}{\errh^{\mu,m}} + \iprd{ \errR^{\mu,m}}{ \dtau \errh^{\phi,m} } 
\nonumber
\\
&\quad+ 2 \iprd{ \nabla \phi^{m} - \nabla \phi^{m-1}}{ \nabla \dtau \errh^{\phi,m} } + 4\iprd{\errh^{\phi,m}}{\dtau \errh^{\phi,m}}  - \iprd{\left(\phi^m\right)^3 - \left(\phi^m\right)^3}{\dtau \errh^{\phi,m} } 
\nonumber
\\
&\quad + 2 \iprd{ \nabla \errP^{\phi,m-1}}{ \nabla \dtau \errh^{\phi,m}  } + \aIPh{\phi^m}{\dtau \errh^{\phi,m} - E_h \dtau \errh^{\phi,m} } 
\nonumber
\\
&\quad+\iprd{\left(\phi^m\right)^3 + (1-\epsilon)\phi^m}{\dtau \errh^{\phi,m} - E_h \dtau \errh^{\phi,m} } - 2 \iprd{ \nabla \phi^{m}}{ \nabla \dtau \errh^{\phi,m}  - \nabla E_h \dtau \errh^{\phi,m} } 
\nonumber
\\
&\quad- \iprd{ \mu^m}{ \dtau \errh^{\phi,m} - E_h \dtau \errh^{\phi,m} }.
\end{align}

The next lemma relates the discrete negative norm of $\dtau \errh^{\phi,m}$ to the $L^2$ norm of $\nabla \errh^{\mu,m}$ and is critical to the proof of the main theorem of the paper which is stated below.

\begin{lemma}\label{lem:minus-1-norm-bound}
Let $(\phi^{m}, \mu^{m})$ be a weak solution to \eqref{eq:pfc-weak}, with the additional regularities \eqref{eq:higher-regularities}. Then for any $h, \tau >0$ and any $0 \le m \le M$, we have
\begin{align*}
\norm{\dtau \errh^{\phi,m}}{-1,h}^2 \le 4 \norm{\nabla \errh^{\mu,m}}{L^2}^2 + C \tau \int_{t_{m-1}}^{t_{m}}\norm{\partial_{ss}\phi(s)}{L^2}^2 ds +  \frac{C}{\tau} \int_{t_{m-1}}^{t_{m}} \norm{ \partial_s \phi(s) -  P_h \partial_s \phi(s) }{2,h}^2 ds ,
\end{align*}
where the constant $C$ may depend upon a Poincar\'{e} constant but does not depend on the spacial step size $h$ or the time step size $\tau$.
\end{lemma}

\begin{proof}
The proof is similar to that of Lemma 3.5 in \cite{DFW:15:CHDS}. Details of the proof can be found in Appendix \ref{app:-1,h_bound_proof}.
\end{proof}

The following lemma will bound many of the terms on the right hand side of \eqref{eq:main-error} by oscillations in the chemical potential $\mu$ which is considered data. The procedure is known as a medius analysis and has been utilized in much of the literature found on the C$^0$-IP method and details can be found in \cite{brenner:11:frontiers}. However, it's application to time dependent problems is new. We provide the key aspects of the proof below but reserve several of the more rigorous details for Appendix \ref{app:osc-bound-proof}.  

\begin{lemma}\label{lem:bounds-oscillations}
Suppose $(\phi^{m},\mu^{m})$ is a weak solution to \eqref{eq:pfc-weak}, with the additional regularities \eqref{eq:higher-regularities}. Then for any $h,\tau > 0$ and any $0 \le m \le M$,
\begin{align}
& \aIPh{\phi^m}{\errh^{\phi,m} - E_h \errh^{\phi,m}} + \iprd{\left(\phi^m\right)^3 +(1-\epsilon)\phi^m}{ \errh^{\phi,m} - E_h \errh^{\phi,m} }
\nonumber
\\
&-2 \iprd{\nabla \phi^{m} }{ \nabla\left(\errh^{\phi,m} -  E_h \errh^{\phi,m}\right)}- \iprd{\mu^m}{\errh^{\phi,m} - E_h \errh^{\phi,m}} 
\nonumber
\\
&\hspace{2in}\le C \left[\text{Osc}_j(\mu^m)\right]^2 +  C  \norm{\errP^{\phi,m}}{2,h}^2 + \frac{C_{coer}}{4\beta} \norm{\errh^{\phi,m}}{2,h}^2
\label{eq:error-bound-2-semi}
\\
&\text{and}
\nonumber
\\
& \aIPh{\dtau \phi^{m}}{\errh^{\phi,m-1} - E_h \errh^{\phi,m-1}} + \iprd{\dtau \left(\left(\phi^m\right)^3 + (1-\epsilon)\phi^m \right)}{\errh^{\phi,m-1} - E_h \errh^{\phi,m-1} } 
\nonumber
\\
&-2 \iprd{\dtau \nabla \phi^{m}}{\nabla \left(\errh^{\phi,m-1} -  E_h \errh^{\phi,m-1}\right)}- \iprd{\dtau \mu^{m}}{\errh^{\phi,m-1} - E_h \errh^{\phi,m-1}} 
\nonumber
\\
&\hspace{2in}\le C \left[\text{Osc}_j(\partial_t \mu(t^*))\right]^2 +  C \norm{\errP^{\phi,m}}{2,h}^2 + C \norm{\errh^{\phi,m-1}}{2,h}^2
\label{eq:error-bound-3-semi}
\end{align}
for $t^* \in (t_{m}, t_{m+1})$ where the arbitrary constant $\beta > 0$ and where $\text{Osc}_j(\nu)$ is referred to as the oscillation of $\nu$ (of order $j$) defined by 
\begin{align} \label{eq:data-osc}
\text{Osc}_j(\nu):= \left( \sum\limits_{K \in \TT_h} h^4 \norm{\nu - \tilde{\nu}}{L^2(K)}^2 \right)^{\frac{1}{2}}
 \end{align}
and where $ \tilde{\nu}$ is the $L^2$ orthogonal projection of $\nu$ on $P_j(\Omega, \TT_h)$, the space of piecewise polynomial functions of degree less than or equal to $j$, i.e.,
\begin{align*}
\int_\Omega (\nu -  \tilde{\nu}) \psi \, dx = 0 \quad \forall \, \psi \in P_j(\Omega, \TT_h).
\end{align*}
\end{lemma}

\begin{proof}
The definition of the Ritz projection \eqref{eq:H2-Ritz-projection} leads to,  
\begin{align}
&\aIPh{\phi^m}{\errh^{\phi,m} -E_h \errh^{\phi,m}} + \iprd{\left(\phi^m\right)^3 +(1-\epsilon)\phi^m}{ \errh^{\phi,m} - E_h \errh^{\phi,m} } 
\nonumber
\\
&\quad- 2 \iprd{\nabla \phi^{m} }{ \nabla\left(\errh^{\phi,m} -  E_h \errh^{\phi,m}\right)} - \iprd{\mu^m}{\errh^{\phi,m} -E_h \errh^{\phi,m}}
\nonumber
\\
& = \aIPh{\phi^m - P_h \phi^m}{\errh^{\phi,m} -E_h \errh^{\phi,m}} + \aIPh{P_h \phi^m}{\errh^{\phi,m} -E_h \errh^{\phi,m}} 
\nonumber
\\
&\quad + \iprd{\left(\phi^m\right)^3 - \left(P_h \phi^m\right)^3 }{ \errh^{\phi,m} - E_h \errh^{\phi,m} } + \iprd{ \left(P_h \phi^m\right)^3 }{ \errh^{\phi,m} - E_h \errh^{\phi,m} } 
\nonumber
\\
&\quad + (1-\epsilon) \iprd{\phi^m - P_h \phi^m}{ \errh^{\phi,m} - E_h \errh^{\phi,m} } + (1-\epsilon) \iprd{P_h \phi^m}{ \errh^{\phi,m} - E_h \errh^{\phi,m} } 
\nonumber
\\
&\quad - 2 \iprd{\nabla (\phi^{m} - P_h \phi^m) }{ \nabla\left(\errh^{\phi,m} -  E_h \errh^{\phi,m}\right)} - 2 \iprd{\nabla ( P_h \phi^m) }{ \nabla\left(\errh^{\phi,m} -  E_h \errh^{\phi,m}\right)} 
\nonumber
\\
&\quad - \iprd{\mu^m}{\errh^{\phi,m} -E_h \errh^{\phi,m}}
\nonumber
\\
& = - \aIPh{\phi^m - P_h \phi^m}{E_h \errh^{\phi,m}} + \aIPh{P_h \phi^m}{\errh^{\phi,m} -E_h \errh^{\phi,m}} 
\nonumber
\\
&\quad + \iprd{\left(\phi^m\right)^3 - \left(P_h \phi^m\right)^3 }{ \errh^{\phi,m} - E_h \errh^{\phi,m} } + \iprd{ \left(P_h \phi^m\right)^3 }{ \errh^{\phi,m} - E_h \errh^{\phi,m} } 
\nonumber
\\
&\quad - (1-\epsilon) \iprd{\phi^m - P_h \phi^m}{ E_h \errh^{\phi,m} } + (1-\epsilon) \iprd{P_h \phi^m}{ \errh^{\phi,m} - E_h \errh^{\phi,m} } 
\nonumber
\\
&\quad - 2 \iprd{\nabla (\phi^{m} - P_h \phi^m) }{ \nabla\left(\errh^{\phi,m} -  E_h \errh^{\phi,m}\right)} - 2 \iprd{\nabla ( P_h \phi^m) }{ \nabla\left(\errh^{\phi,m} -  E_h \errh^{\phi,m}\right)} 
\nonumber
\\
&\quad - \iprd{\mu^m}{\errh^{\phi,m} -E_h \errh^{\phi,m}}
\nonumber
\\
&= \aIPh{P_h \phi^m}{\errh^{\phi,m} -E_h \errh^{\phi,m}} + \iprd{ \left(P_h \phi^m\right)^3 + (1-\epsilon) P_h \phi^m}{ \errh^{\phi,m} - E_h \errh^{\phi,m} } 
\nonumber
\\
&\quad - 2 \iprd{\nabla ( P_h \phi^m) }{ \nabla\left(\errh^{\phi,m} -  E_h \errh^{\phi,m}\right)} - \iprd{\mu^m}{\errh^{\phi,m} -E_h \errh^{\phi,m}} 
\nonumber
\\
&\quad+ \sum_{K \in \TT_h}  \int_K \nabla^2 (P_h \phi^m - \phi^m) : \nabla^2 (E_h \errh^{\phi,m}) dx +  \sum_{e \in \Eh} \int_e \dgal[\Bigg]{\frac{\partial^2  \left(E_h \errh^{\phi,m}\right)}{\partial n_e^2} } \left\llbracket \frac{\partial P_h \phi^{m}}{\partial n_e} \right\rrbracket dS 
\nonumber
\\
&\quad- (1-\epsilon) \iprd{\phi^m - P_h \phi^m}{ E_h \errh^{\phi,m} } 
\nonumber
\\
&\quad + \iprd{\left(\phi^m\right)^3 - \left(P_h \phi^m\right)^3 }{ \errh^{\phi,m} - E_h \errh^{\phi,m} }  - 2 \iprd{\nabla (\phi^{m} - P_h \phi^m) }{ \nabla\left(\errh^{\phi,m} -  E_h \errh^{\phi,m}\right)}.
\label{eq:aiph-split}
\end{align}

Furthermore, the following equivalent formulation of the bilinear form $\aIPh{\cdot}{\cdot}$ for functions satisfying $w \in H^4(\Omega,\TT_h) \cap H^1(\Omega)$ and $v \in H^2(\Omega, \TT_h) \cap H^1(\Omega)$:
\begin{align*}
&\aIPh{w}{v} := \, \sum_{K \in \TT_h}  \int_K \left(\Delta^2 w\right) v \, dx + \sum_{e \in \Eh} \int_e \dgal[\Bigg]{\frac{\partial^2 v}{\partial n_e^2} } \left\llbracket \frac{\partial w}{\partial n_e} \right\rrbracket dS 
\nonumber
\\
&\quad+ \sum_{e \in \Eh} \int_e \left( \left\llbracket \frac{\partial \Delta w}{\partial n_e} \right\rrbracket v -  \left\llbracket \frac{\partial^2 w}{\partial n_e^2} \right\rrbracket \dgal[\Bigg]{\frac{\partial v}{\partial n_e} } - \left\llbracket \frac{\partial^2 w}{\partial n_e \partial t_e}  \right\rrbracket \frac{\partial v}{\partial t_e} + \frac{\alpha}{|e|} \left\llbracket \frac{\partial w}{\partial n_e} \right\rrbracket  \left\llbracket \frac{\partial v}{\partial n_e} \right\rrbracket \right) dS,
\end{align*}
where $H^s(\Omega,\TT_h) := \{v \in L^2(\Omega) | v_K \in H^s(K) \forall K \in \TT_h\}$ and where $t_e$ denotes the unit counterclockwise tangent vector, yields the following
\begin{align}\label{eq:aiph-with-projection}
& \aIPh{P_h \phi^{m}}{ \errh^{\phi,m} - E_h \errh^{\phi,m}} = \, \sum_{K \in \TT_h}  \int_K \left(\Delta^2 P_h \phi^m \right) \left(\errh^{\phi,m} -E_h \errh^{\phi,m}\right) \, dx 
\nonumber
\\
&\quad+ \sum_{e \in \Eh} \int_e \left( \dgal[\Bigg]{\frac{\partial^2  \left(\errh^{\phi,m} -E_h \errh^{\phi,m}\right)}{\partial n_e^2} } \left\llbracket \frac{\partial P_h \phi^{m}}{\partial n_e} \right\rrbracket dS + \left\llbracket \frac{\partial \Delta P_h \phi^{m}}{\partial n_e} \right\rrbracket \left(\errh^{\phi,m} -E_h \errh^{\phi,m}\right) \right) dS
\nonumber
\\
&\quad- \sum_{e \in \Eh} \int_e \left( \left\llbracket \frac{\partial^2 P_h \phi^m}{\partial n_e^2} \right\rrbracket \dgal[\Bigg]{\frac{\partial  \left(\errh^{\phi,m} -E_h \errh^{\phi,m}\right)}{\partial n_e} } dS -   \left\llbracket \frac{\partial^2 P_h \phi^m}{\partial n_e \partial t_e} \right\rrbracket \frac{\partial  \left(\errh^{\phi,m} -E_h \errh^{\phi,m}\right)}{\partial t_e} \right) dS
\nonumber
\\
&\quad+ \alpha \sum_{e \in \Eh} \frac{1}{|e|} \int_e  \left\llbracket \frac{\partial P_h \phi^m}{\partial n_e} \right\rrbracket  \left\llbracket \frac{\partial \left(\errh^{\phi,m} -E_h \errh^{\phi,m}\right)}{\partial n_e} \right\rrbracket dS.
\end{align}

Combining equations \eqref{eq:aiph-split}--\eqref{eq:aiph-with-projection}, we have
\begin{align*}
&\aIPh{\phi^m}{\errh^{\phi,m} -E_h \errh^{\phi,m}} + \iprd{\left(\phi^m\right)^3 +(1-\epsilon)\phi^m}{ \errh^{\phi,m} - E_h \errh^{\phi,m} } 
\\
&\quad- 2 \iprd{\nabla \phi^{m} }{ \nabla\left(\errh^{\phi,m} -  E_h \errh^{\phi,m}\right)} - \iprd{\mu^m}{\errh^{\phi,m} -E_h \errh^{\phi,m}}
\\
&= \sum_{K \in \TT_h}  \int_K \left(\Delta^2 P_h \phi^m + \left(P_h \phi^m\right)^3+(1-\epsilon)P_h \phi^m + 2 \Delta P_h \phi^{m}- \mu^m\right) \left(\errh^{\phi,m} -E_h \errh^{\phi,m}\right) \, dx 
\\
&\quad+ \sum_{e \in \Eh} \int_e \left( \dgal[\Bigg]{ \frac{\partial^2  \errh^{\phi,m}}{\partial n_e^2} } \left\llbracket \frac{\partial P_h \phi^{m}}{\partial n_e} \right\rrbracket + \left\llbracket \frac{\partial \Delta P_h \phi^{m}}{\partial n_e} \right\rrbracket \left(\errh^{\phi,m} -E_h \errh^{\phi,m}\right) \right) dS
\\
&\quad- \sum_{e \in \Eh} \int_e \left(  \left\llbracket \frac{\partial^2 P_h \phi^m}{\partial n_e^2} \right\rrbracket \dgal[\Bigg]{\frac{\partial  \left(\errh^{\phi,m} -E_h \errh^{\phi,m}\right)}{\partial n_e} } + \left\llbracket \frac{\partial^2 P_h \phi^m}{\partial n_e \partial t_e} \right\rrbracket \frac{\partial  \left(\errh^{\phi,m} -E_h \errh^{\phi,m}\right)}{\partial t_e} \right) dS
\\
&\quad+ \alpha \sum_{e \in \Eh} \frac{1}{|e|} \int_e  \left\llbracket \frac{\partial P_h \phi^m}{\partial n_e} \right\rrbracket  \left\llbracket \frac{\partial \left(\errh^{\phi,m} -E_h \errh^{\phi,m}\right)}{\partial n_e} \right\rrbracket dS
\\
&\quad+\sum_{K \in \TT_h}  \int_K \nabla^2 (P_h \phi^m - \phi^m) : \nabla^2 (E_h \errh^{\phi,m}) dx - (1-\epsilon) \iprd{\phi^m - P_h \phi^m}{ E_h \errh^{\phi,m} } 
\nonumber
\\
&\quad + \iprd{\left(\phi^m\right)^3 - \left(P_h \phi^m\right)^3 }{ \errh^{\phi,m} - E_h \errh^{\phi,m} }  - 2 \iprd{\nabla (\phi^{m} - P_h \phi^m) }{ \nabla\left(\errh^{\phi,m} -  E_h \errh^{\phi,m}\right)}.
\end{align*}
Following the medius analysis presented in \cite{brenner:11:frontiers} (details may be found in Appendix \ref{app:osc-bound-proof}), all but the last three terms on the right-hand side can be bounded by:
\begin{align*}
 \Bigg( \left[\text{Osc}_j(\mu^m)\right]^2 + \sum_{K \in \TT_h} |\phi^m - P_h \phi^m |_{H^2(K)}^2 + \sum_{e \in \Eh} \frac{1}{|e|} \norm{\left\llbracket \frac{\partial (P_h \phi^{m} - \phi^m)}{\partial n_e} \right\rrbracket }{L^2(e)}^2\Bigg)^{\nicefrac{1}{2}} .
\end{align*}
The last three terms on the right-hand side are bounded as follows:
\begin{align*}
&\left|(1-\epsilon) \iprd{ \phi^m - P_h \phi^m }{ E_h \errh^{\phi,m} }\right| \le C \left(\sum_{K \in \TT_h} |\phi^m - P_h \phi^m|_{L^2(K)}^2\right)^{\nicefrac{1}{2}} \norm{\errh^{\phi,m}}{2,h} 
\\
&\quad \le C \left(\sum_{K \in \TT_h} |\phi^m - P_h \phi^m|_{H^2(K)}^2\right)^{\nicefrac{1}{2}} \norm{\errh^{\phi,m}}{2,h} 
\\
 & \left|\iprd{\left(\phi^m\right)^3 - \left(P_h \phi^m\right)^3 }{ \errh^{\phi,m} - E_h \errh^{\phi,m} } \right| \le \norm{\left(\phi^m\right)^3 - \left(P_h \phi^m\right)^3}{L^2} \norm{ \errh^{\phi,m} - E_h \errh^{\phi,m} }{L^2}
 \\
 &\quad\le C \norm{\left(\phi^m\right)^2 + \phi^m P_h \phi^m + \left(P_h \phi^m\right)^2}{L^3} \norm{\phi^m - P_h \phi^m}{L^6} \norm{ \errh^{\phi,m} - E_h \errh^{\phi,m} }{2,h}
  \\
 &\quad\le C\left( \norm{\left(\phi^m\right)^2}{L^3} + \norm{\left(P_h \phi^m\right)^2}{L^3} \right) \norm{\phi^m - P_h \phi^m}{L^6} \norm{ \errh^{\phi,m} - E_h \errh^{\phi,m} }{2,h}
   \\
 &\quad\le C\left( \norm{\phi^m}{L^6}^2 + \norm{P_h \phi^m}{L^6}^2 \right) \norm{\phi^m - P_h \phi^m}{L^6} \norm{ \errh^{\phi,m} - E_h \errh^{\phi,m} }{2,h}
    \\
 &\quad\le C\left( \norm{\phi^m}{H^1}^2 + \norm{P_h \phi^m}{H^1}^2 \right) \norm{\phi^m - P_h \phi^m}{H^1} \norm{ \errh^{\phi,m} - E_h \errh^{\phi,m} }{2,h}
     \\
 &\quad\le C\left( \norm{\phi^m}{H^1}^2 \right) \norm{\phi^m - P_h \phi^m}{H^1} \norm{ \errh^{\phi,m} - E_h \errh^{\phi,m} }{2,h}
\\
&\quad\le C \left(\sum_{K \in \TT_h} |\phi^m - P_h \phi^m|_{H^2(K)}^2\right)^{\nicefrac{1}{2}} \norm{\errh^{\phi,m}}{2,h}
\\
& \left|2 \iprd{ \nabla \left(\phi^{m} - P_h \phi^m \right)}{ \nabla\left(\errh^{\phi,m} -  E_h \errh^{\phi,m} \right) } \right| \le C \left(\sum_{K \in \TT_h} |\phi^m - P_h \phi^m|_{H^2(K)}^2\right)^{\nicefrac{1}{2}} \norm{\errh^{\phi,m}}{2,h} ,
\end{align*}
where we have assumed that the Ritz projection \eqref{eq:H2-Ritz-projection} is stable with respect to the $H^1$ norm and that $\phi \in L^{\infty}(0,T;H^3(\Omega))$ giving that $\norm{\phi^m}{H^1} \le C$ for any $0 \le m \le M$ and where we have used the Cauchy-Schwarz and H\"{o}lder's inequalities.

Thus, we have
\begin{align*}
&\aIPh{ \phi^m}{\errh^{\phi,m} -E_h \errh^{\phi,m}} + \iprd{\left( \phi^m\right)^3 +(1-\epsilon)\phi^m}{ \errh^{\phi,m} - E_h \errh^{\phi,m} } 
\\
&\quad- 2 \iprd{\nabla \phi^{m} }{ \nabla\left(\errh^{\phi,m} -  E_h \errh^{\phi,m}\right)} - \iprd{\mu^m}{\errh^{\phi,m} -E_h \errh^{\phi,m}}
\\
&\quad \le C \Bigg( \left[\text{Osc}_j(\mu^m)\right]^2 + \sum_{K \in \TT_h} |\phi^m - P_h \phi^m |_{H^2(K)}^2 + \sum_{e \in \Eh} \frac{1}{|e|} \norm{\left\llbracket \frac{\partial (P_h \phi^{m} - \phi^m)}{\partial n_e} \right\rrbracket }{L^2(e)}^2\Bigg)^{\nicefrac{1}{2}} 
\\
&\qquad  \times \norm{\errh^{\phi,m}}{2,h} 
\\
&\quad  \le C \left(\left[\text{Osc}_j(\mu^m)\right] + \norm{ \phi^{m} - P_h \phi^m}{2,h} \right) \norm{\errh^{\phi,m}}{2,h} .
\end{align*}

Equation \eqref{eq:error-bound-2-semi} follows from an application of Young's inequality. Finally, a similar strategy along with an application of the Mean Value Theorem yields \eqref{eq:error-bound-3-semi}.
\end{proof}

We are now in position to prove the main theorem in this section.

\begin{theorem}\label{thm:error-bound-final}
Suppose $(\phi^m,\mu^m)$ is a weak solution to \eqref{eq:pfc-weak}, with the additional regularities \eqref{eq:higher-regularities}. Then for any $\tau ,h > 0,  \epsilon <\min\{1 + \frac{(1-\gamma)C_{coer} - 1}{C_{coer}},1 + 4 - \frac{4(1-\gamma)}{C_{coer}}\} $ and any $0 \le \tau \le M$, 
\begin{align}\label{eq:error-bound-final}
&\norm{\errh^{\phi,\ell}}{2, h}^2 + C \norm{\errh^{\phi,\ell}}{L^2}^2 + C \tau \sum\limits_{m=1}^{\ell}\norm{\nabla \errh^{\mu,m}}{L^2}^2  
\nonumber
\\
&\quad + C \tau^2 \sum\limits_{m=1}^{\ell} \left[ \norm{\dtau \errh^{\phi,\ell}}{2, h}^2 + (1-\epsilon)\norm{\dtau \errh^{\phi,m}}{L^2}^2 + \norm{\nabla \dtau \errh^{\phi,m}}{L^2}^2\right] \le C^* (h^2 + \tau^2 ).
\end{align}
where $C^*$ may depend on the oscillations of $\mu$ and $\partial_t \mu$ and the final stopping time $T$ but does not depend upon the spacial step size $h$ or the time step size $\tau$.
\end{theorem}

\begin{proof}
Starting with the first five terms on the right hand side of \eqref{eq:main-error} and using Young's and H\"{o}lder's inequalities, Poincar\'{e}'s inequality, Taylor's theorem, and properties \eqref{eq:plus-1-minus-1-estimate} and \eqref{eq:plus-1-minus-1-estimate-g}, we have
\begin{align}
\iprd{\dtau \phi^{m} - \partial_t \phi^{m}}{\errh^{\mu,m}}&\le \norm{\dtau \phih^{m} - \partial_t \phi^{m}}{L^2}\norm{ \errh^{\mu,m}}{L^2}^2
\nonumber
\\
&\le C  \tau \int_{t_{m-1}}^{t_{m}}\norm{\partial_{ss}\phi(s)}{L^2}^2 ds + \frac{1}{12} \norm{\nabla \errh^{\mu,m}}{L^2}^2,\label{eq:error-bound-1}
\end{align}
\begin{align}
\iprd{\dtau \errP^{\phi,m}}{\errh^{\mu,m}} &\le \norm{\dtau \errP^{\phi,m}}{L^2} \norm{\errh^{\mu,m}}{L^2}
\nonumber
\\
&\le C \norm{\dtau \errP^{\phi,m}}{L^2}^2 + \frac{1}{12}\norm{\nabla \errh^{\mu,m}}{L^2}^2
\nonumber
\\
&\le \frac{C}{\tau} \int_{t_{m-1}}^{t_{m}} \norm{ P_h \partial_s \phi(s)- \partial_s \phi(s)}{L^2}^2 ds + \frac{1}{12}\norm{\nabla \errh^{\mu,m}}{L^2}^2
\nonumber
\\
&\le \frac{C}{\tau} \int_{t_{m-1}}^{t_{m}} \norm{ \partial_s \phi(s) - P_h \partial_s \phi(s)}{2,h}^2 ds + \frac{1}{12}\norm{\nabla \errh^{\mu,m}}{L^2}^2 ,
\label{eq:error-bound-2}
\\
\iprd{ \errR^{\mu,m}}{ \dtau \errh^{\phi,m} } &\le C\norm{\nabla \errR^{\mu,m}}{L^2}^2 + \frac{1}{48} \norm{\dtau \errh^{\phi,m}}{-1,h}^2,
\label{eq:error-bound-3}
\\
\nonumber
\\
2 \iprd{ \nabla \phi^{m} - \nabla \phi^{m-1}}{\nabla \dtau \errh^{\phi,m} } &= -2 \iprd{\tau \Delta \dtau \phi^m}{ \dtau \errh^{\phi,m} } 
\nonumber
\\
&\le 2\norm{\tau \nabla \Delta \dtau \phi^m}{L^2} \norm{ \dtau \errh^{\phi,m} }{-1,h}
\nonumber
\\
&\le C \tau \int_{t_{m-1}}^{t_{m}} \norm{\partial_s\phi(s)}{H^{3}}^2 \, ds + \frac{1}{48}  \norm{\dtau \errh^{\phi,m}}{-1,h}^2
\label{eq:error-bound-4}
\\
\text{and}
\nonumber
\\
\iprd{ \errh^{\phi,m}}{ \dtau \errh^{\phi,m} }  &\le  C\norm{\errh^{\phi,m}}{2,h}^2 + \frac{1}{48} \norm{\dtau \errh^{\phi,m}}{-1,h}^2.  
\label{eq:error-bound-5}
	\end{align}
For the nonlinear term, we use properties \eqref{eq:plus-1-minus-1-estimate} and \eqref{eq:plus-1-minus-1-estimate-g} along with Lemma \ref{lem:stability} and Young's, H\"{o}lder's, and Poincar\'{e} ineqaulities and the higher-regularities \eqref{eq:higher-regularities} to obtain,
\begin{align}\label{eq:error-bound-6}
&\iprd{\left(\phi^m\right)^3 - \left(\phih^m\right)^3}{\dtau \errh^{\phi,m} } \le \norm{\nabla \left(\left(\phi^m\right)^3 - \left(\phih^m\right)^3\right)}{L^2} \norm{\dtau \errh^{\phi,m}}{-1,h}
\nonumber
\\
&\quad= \norm{3 \left(\phi^m\right)^2 \nabla\phi^m - 3\left(\phih^m\right)^2 \nabla \phih^m}{L^2} \norm{\dtau \errh^{\phi,m}}{-1,h}
\nonumber
\\
&\quad= 3 \norm{\left(\phi^m + \phih^m\right) \nabla \phi^m \err^{\phi,m} + \left(\phih^m\right)^2 \nabla \err^{\phi,m}}{L^2} \norm{\dtau \errh^{\phi,m}}{-1,h}
\nonumber
\\
&\quad\le 3 \left(\norm{\phi^m + \phih^m}{L^6} \norm{\nabla \phi^m}{L^6} \norm{\err^{\phi,m}}{L^6} + \norm{\phih^m}{L^6}^2 \norm{\nabla \err^{\phi,m}}{L^6}\right) \norm{\dtau \errh^{\phi,m}}{-1,h}
\nonumber
\\
&\quad\le C\left(\norm{\nabla \errP^{\phi,m}}{L^2} + \norm{\nabla \errh^{\phi,m}}{L^2} + \norm{ \errP^{\phi,m}}{2,h} + \norm{ \errh^{\phi,m}}{2,h}\right) \norm{\dtau \errh^{\phi,m}}{-1,h}
\nonumber
\\
&\quad\le C \norm{ \errP^{\phi,m}}{2,h}^2 + C\norm{ \errh^{\phi,m}}{2,h}^2 + \frac{1}{48}  \norm{\dtau \errh^{\phi,m}}{-1,h}^2.
\end{align}
	
For the remaining terms, we note that the following discrete product rules hold for any bilinear form and remark that these discrete product rules are key to recovering the predicted error estimates:
\begin{align*}
\iprd{a^{m-1}}{\frac{b^{m} - b^{m-1}}{\tau}} &= \frac{1}{\tau}\left[\iprd{a^{m} }{ b^{m}} - \iprd{a^{m-1}}{b^{m-1}} \right] - \iprd{\frac{ a^{m} - a^{m-1}}{\tau}}{b^{m}}
\\
&= \dtau \iprd{a^{m} }{ b^{m}} - \iprd{\dtau a^{m}}{b^{m}},
\end{align*}
and
\begin{align*}
\iprd{a^{m}}{\frac{b^{m} - b^{m-1}}{\tau}} &= \frac{1}{\tau}\left[\iprd{a^{m} }{ b^{m}} - \iprd{a^{m-1}}{b^{m-1}} \right] - \iprd{\frac{ a^{m} - a^{m-1}}{\tau}}{b^{m-1}}
\\
&= \dtau \iprd{a^{m} }{ b^{m}} - \iprd{\dtau a^{m} }{b^{m-1}}.
\end{align*}
Thus, we have the following bound
\begin{align}
2 \iprd{ \nabla \errP^{\phi,m-1}}{ \nabla \dtau \errh^{\phi,m}  }  &= 2\dtau \iprd{ \nabla \errP^{\phi,m} }{\nabla \errh^{\phi,m}}  - 2 \iprd{\nabla \dtau \errP^{\phi,m}}{\nabla \errh^{\phi,m}}
\nonumber
\\
&\le 2 \dtau\iprd{ \nabla \errP^{\phi,m} }{\nabla \errh^{\phi,m}} + C \norm{\dtau \errP^{\phi,m}}{L^2}^2 + C \norm{ \errh^{\phi,m}}{2,h}^2 
\nonumber
\\
&\le 2 \dtau\iprd{ \nabla \errP^{\phi,m} }{\nabla \errh^{\phi,m}} + \frac{C}{\tau} \int_{t_{m-1}}^{t_{m}} \norm{ \partial_s \phi(s) - P_h \partial_s \phi(s)}{2,h}^2 ds
\nonumber
\\
&\quad+ C \norm{ \errh^{\phi,m}}{2,h}^2 . \label{eq:error-bound-7}
\end{align}

Additionally invoking Lemma \ref{lem:bounds-oscillations} yields,
\begin{align}
&\aIPh{\phi^m}{\dtau \errh^{\phi,m} - E_h \dtau \errh^{\phi,m} } +\iprd{\left(\phi^m\right)^3 + (1-\epsilon)\phi^m}{\dtau \errh^{\phi,m} - E_h \dtau \errh^{\phi,m} } 
\nonumber
\\
&\quad- 2 \iprd{ \nabla \phi^{m}}{ \nabla \dtau \errh^{\phi,m}  - \nabla E_h \dtau \errh^{\phi,m} } - \iprd{ \mu^m}{ \dtau \errh^{\phi,m} - E_h \dtau \errh^{\phi,m} }
\nonumber
\\
&= \dtau \aIPh{\phi^m}{\errh^{\phi,m} - E_h \errh^{\phi,m}} + \dtau \iprd{\left(\phi^m\right)^3 +(1-\epsilon)\phi^m}{ \errh^{\phi,m} - E_h \errh^{\phi,m} }
\nonumber
\\
&\quad-2 \dtau \iprd{\nabla \phi^{m} }{ \nabla\left(\errh^{\phi,m} -  E_h \errh^{\phi,m}\right)} - \dtau \iprd{\mu^m}{\errh^{\phi,m} - E_h \errh^{\phi,m}} 
\nonumber
\\
&\quad-\aIPh{\dtau \phi^{m}}{\errh^{\phi,m-1} - E_h \errh^{\phi,m-1}} - \iprd{\dtau \left(\left(\phi^m\right)^3 + (1-\epsilon)\phi^m \right)}{\errh^{\phi,m-1} - E_h \errh^{\phi,m-1} } 
\nonumber
\\
&\quad+2 \iprd{\dtau \nabla \phi^{m}}{\nabla \left(\errh^{\phi,m-1} -  E_h \errh^{\phi,m-1}\right)} + \iprd{\dtau \mu^{m}}{\errh^{\phi,m-1} - E_h \errh^{\phi,m-1}} 
\nonumber
\\
&\le \dtau \aIPh{\phi^m}{\errh^{\phi,m} - E_h \errh^{\phi,m}} + \dtau \iprd{\left(\phi^m\right)^3 +(1-\epsilon)\phi^m}{ \errh^{\phi,m} - E_h \errh^{\phi,m} }
\nonumber
\\
&\quad-2 \dtau \iprd{\nabla \phi^{m} }{ \nabla\left(\errh^{\phi,m} -  E_h \errh^{\phi,m}\right)} - \dtau \iprd{\mu^m}{\errh^{\phi,m} - E_h \errh^{\phi,m}} 
\nonumber
\\
&\quad+ C \left[\text{Osc}_j(\partial_t \mu(t^*))\right]^2 +  C \norm{\errP^{\phi,m}}{2,h}^2 + C\norm{\errh^{\phi,m-1}}{2,h}^2. \label{eq:error-bound-8}
\end{align}

Now applying the polarization property to the appropriate terms on the left-hand side of \eqref{eq:main-error} and combining the resulting inequality with equations \eqref{eq:error-bound-1}--\eqref{eq:error-bound-8}, we have
\begin{align*}
&\norm{\nabla \errh^{\mu,m}}{L^2}^2 + \frac{1}{2} \dtau \, \aIPh{\errh^{\phi,m}}{\errh^{\phi,m}} + \frac{\tau}{2} \aIPh{\dtau \errh^{\phi,m}}{\dtau \errh^{\phi,m}} 
\\
&\quad + \frac{(5-\epsilon)}{2} \dtau \norm{\errh^{\phi,m}}{L^2}^2  + \frac{(5-\epsilon)\tau}{2} \norm{\dtau \errh^{\phi,m}}{L^2}^2 + \tau \norm{\nabla \dtau \errh^{\phi,m}}{L^2}^2
\\
&\le \dtau \iprd{\nabla \errh^{\phi,m} }{\nabla \errh^{\phi,m}} + 2 \dtau\iprd{ \nabla \errP^{\phi,m} }{\nabla \errh^{\phi,m}}  + \frac{2}{12} \norm{\nabla \errh^{\mu,m}}{L^2}^2 + \frac{1}{12} \norm{\dtau \errh^{\phi,m}}{-1,h}^2 
\\
&\quad +  C\norm{\errh^{\phi,m-1}}{2,h}^2 + C \norm{\errh^{\phi,m}}{2,h}^2   + C \norm{\nabla \errR^{\mu,m}}{L^2}^2  + C \norm{\errP^{\phi,m}}{2,h}^2 +  C \left[\text{Osc}_j(\partial_t \mu(t^*))\right]^2 
\\
&\quad + C  \tau \int_{t_{m-1}}^{t_{m}} \left[ \norm{\partial_s\phi(s)}{L^2}^2 + \norm{\partial_{ss}\phi(s)}{L^2}^2  \right] ds  + \frac{C}{\tau} \int_{t_{m-1}}^{t_{m}} \norm{\partial_s \phi(s) - P_h \partial_s \phi(s)}{2,h}^2 ds 
\\
&\quad+\dtau \aIPh{\phi^m}{\errh^{\phi,m} - E_h \errh^{\phi,m}} + \dtau \iprd{\left(\phi^m\right)^3 +(1-\epsilon)\phi^m}{ \errh^{\phi,m} - E_h \errh^{\phi,m} }
\nonumber
\\
&\quad-2 \dtau \iprd{\nabla \phi^{m} }{ \nabla\left(\errh^{\phi,m} -  E_h \errh^{\phi,m}\right)} - \dtau \iprd{\mu^m}{\errh^{\phi,m} - E_h \errh^{\phi,m}} .
\end{align*}

Invoking Lemma \ref{lem:minus-1-norm-bound}, we have
\begin{align*}
&\norm{\nabla \errh^{\mu,m}}{L^2}^2 + \frac{1}{2} \dtau \, \aIPh{\errh^{\phi,m}}{\errh^{\phi,m}} + \frac{\tau}{2} \aIPh{\dtau \errh^{\phi,m}}{\dtau \errh^{\phi,m}} 
\\
&\quad+ \frac{(5-\epsilon)}{2} \dtau \norm{\errh^{\phi,m}}{L^2}^2  + \frac{(5-\epsilon)\tau}{2}  \norm{\dtau \errh^{\phi,m}}{L^2}^2 + \tau \norm{\nabla \dtau \errh^{\phi,m}}{L^2}^2
\\
&\le\dtau \iprd{\nabla \errh^{\phi,m} }{\nabla \errh^{\phi,m}}  + 2 \dtau\iprd{ \nabla \errP^{\phi,m} }{\nabla \errh^{\phi,m}}  + \frac{6}{12} \norm{\nabla \errh^{\mu,m}}{L^2}^2 + C \norm{\errh^{\phi,m}}{2,h}^2   
\\
&\quad +  C\norm{\errh^{\phi,m-1}}{2,h}^2 +  C \norm{\nabla \errR^{\mu,m}}{L^2}^2  +  C \norm{\errP^{\phi,m}}{2,h}^2  +  C \left[\text{Osc}_j(\partial_t \mu(t^*))\right]^2 
\\
&\quad + C  \tau \int_{t_{m-1}}^{t_{m}} \left[ \norm{\partial_s\phi(s)}{L^2}^2 + \norm{\partial_{ss}\phi(s)}{L^2}^2  \right] ds  + \frac{C}{\tau} \int_{t_{m-1}}^{t_{m}} \norm{\partial_s \phi(s) - P_h \partial_s \phi(s)}{2,h}^2 ds 
\\
&\quad+\dtau \aIPh{\phi^m}{\errh^{\phi,m} - E_h \errh^{\phi,m}} + \dtau \iprd{\left(\phi^m\right)^3 +(1-\epsilon)\phi^m}{ \errh^{\phi,m} - E_h \errh^{\phi,m} }
\nonumber
\\
&\quad-2 \dtau \iprd{\nabla \phi^{m} }{ \nabla\left(\errh^{\phi,m} -  E_h \errh^{\phi,m}\right)} - \dtau \iprd{\mu^m}{\errh^{\phi,m} - E_h \errh^{\phi,m}} .
\end{align*}
Combining like terms, applying $2\tau \sum\limits_{m=1}^{\ell}$, using the fact that $\errh^{\phi,0} = 0$, invoking equation \eqref{eq:L4-bound} and Lemma \ref{lem:grad-split}, and applying H\"{o}lder's inequality, we obtain
\begin{align*}
&\aIPh{\errh^{\phi,\ell}}{\errh^{\phi,\ell}} + (5-\epsilon) \norm{\errh^{\phi,\ell}}{L^2}^2 + \tau \sum\limits_{m=1}^{\ell}\norm{\nabla \errh^{\mu,m}}{L^2}^2  
\\
&\quad + \tau^2 \sum\limits_{m=1}^{\ell} \left[ \aIPh{\dtau \errh^{\phi,m}}{\dtau \errh^{\phi,m}} + (5-\epsilon)\norm{\dtau \errh^{\phi,m}}{L^2}^2 + 2 \norm{\nabla \dtau \errh^{\phi,m}}{L^2}^2\right]
\\
&\le 2\sqrt{(1-\gamma)}\norm{ \errh^{\phi,\ell} }{L^2} \norm{\errh^{\phi,\ell}}{2,h}  + 4\sqrt{(1-\gamma)}\norm{\errP^{\phi,\ell}}{L^2}\norm{\errh^{\phi,\ell}}{2,h} + C \tau \sum\limits_{m=1}^{\ell} \norm{\errh^{\phi,m}}{2,h}^2 
\\
&\quad+C \tau \sum\limits_{m=1}^{\ell}\left[ \norm{\nabla \errR^{\mu,m}}{L^2}^2 +  \norm{\errP^{\phi,m}}{2,h}^2  + \left[\text{Osc}_j(\partial_t \mu(t^*))\right]^2 \right]
\\
&\quad+ C  \tau^2 \int_{t_{0}}^{t_{\ell}} \left[ \norm{\partial_s\phi(s)}{L^2}^2 + \norm{\partial_{ss}\phi(s)}{L^2}^2   \right] ds  + C \int_{t_{0}}^{t_{\ell}} \norm{\partial_s \phi(s) - P_h \partial_s \phi(s)}{2,h}^2 ds
\\
&\quad+ 2 \Big[ \aIPh{\phi^{\ell}}{ \errh^{\phi,\ell} - E_h \errh^{\phi,\ell}} +  \iprd{\left(\phi^\ell\right)^3 + (1-\epsilon)\phi^\ell }{ \errh^{\phi,\ell} - E_h \errh^{\phi,\ell} } 
\\
&\quad- 2 \iprd{\nabla \phi^{\ell} }{ \nabla\left(\errh^{\phi,\ell} -  E_h \errh^{\phi,\ell}\right)}  -  \iprd{\mu^{\ell}}{\errh^{\phi,\ell} - E_h \errh^{\phi,\ell}} \Big].
\end{align*}
Applying Young's and H\"{o}lder's inequalities and Lemma \ref{lem:bounds-oscillations} yields, 
\begin{align*}
&\aIPh{\errh^{\phi,\ell}}{\errh^{\phi,\ell}} + (5-\epsilon) \norm{\errh^{\phi,\ell}}{L^2}^2 + \tau \sum\limits_{m=1}^{\ell}\norm{\nabla \errh^{\mu,m}}{L^2}^2  
\\
&\quad + \tau^2 \sum\limits_{m=1}^{\ell} \left[ \aIPh{\dtau \errh^{\phi,m}}{\dtau \errh^{\phi,m}} + (5-\epsilon)\norm{\dtau \errh^{\phi,m}}{L^2}^2 + 2 \norm{\nabla \dtau \errh^{\phi,m}}{L^2}^2\right]
\\
&\le \frac{C_{coer}}{2\beta}\norm{\errh^{\phi,\ell}}{2,h}^2 + \frac{4(1-\gamma)\beta}{C_{coer}} \norm{\errh^{\phi,\ell}}{L^2}^2 + C \norm{\errP^{\phi,\ell}}{2,h}^2 + C \tau \sum\limits_{m=1}^{\ell} \norm{\errh^{\phi,m}}{2,h}^2 
\\
&\quad+C \tau \sum\limits_{m=1}^{\ell}\left[ \norm{\nabla \errR^{\mu,m}}{L^2}^2 +  \norm{\errP^{\phi,m}}{2,h}^2  + \left[\text{Osc}_j(\partial_t \mu(t^*))\right]^2 \right]
\\
&\quad+ C  \tau^2 \int_{t_{0}}^{t_{\ell}} \left[ \norm{\partial_s\phi(s)}{L^2}^2 + \norm{\partial_{ss}\phi(s)}{L^2}^2   \right] ds  + C \int_{t_{0}}^{t_{\ell}} \norm{\partial_s \phi(s) - P_h \partial_s \phi(s)}{2,h}^2 ds
\\
&\quad+ 2 \Big[ C \left[\text{Osc}_j(\mu^\ell)\right]^2 +  C \norm{\errP^{\phi,\ell}}{2,h}^2 + \frac{C_{coer}}{4\beta} \norm{\errh^{\phi,\ell}}{2,h}^2\Big],
\end{align*}
for $t^* \in (t_{m-1}, t_{m})$. Invoking Lemma \ref{lem:aIPh-boundedness} and combining like terms, we have
\begin{align*}
&\left[C_{coer} -\frac{C_{coer}}{\beta} \right] \norm{\errh^{\phi,\ell}}{2, h}^2 + \left[(5-\epsilon) - \frac{4(1-\gamma)}\beta{C_{coer}} \right] \norm{\errh^{\phi,\ell}}{L^2}^2 + C \tau \sum\limits_{m=1}^{\ell}\norm{\nabla \errh^{\mu,m}}{L^2}^2  
\nonumber
\\
&\quad+ C \tau^2 \sum\limits_{m=1}^{\ell} \left[ C_{coer} \norm{\dtau \errh^{\phi,\ell}}{2, h}^2 + (5-\epsilon)\norm{\dtau \errh^{\phi,m}}{L^2}^2 + \norm{\nabla \dtau \errh^{\phi,m}}{L^2}^2\right] 
\nonumber
\\
&\le\,  C \tau \sum\limits_{m=1}^{\ell }  \norm{\errh^{\phi,m}}{2,h}^2 + C \norm{\errP^{\phi,\ell}}{2,h}^2 + C \left[\text{Osc}_j(\mu^\ell)\right]^2 
\\
&\quad+C \tau \sum\limits_{m=1}^{\ell}\left[ \norm{\nabla \errR^{\mu,m}}{L^2}^2 +  \norm{\errP^{\phi,m}}{2,h}^2  + \left[\text{Osc}_j(\partial_t \mu(t^*))\right]^2 \right]
\\
&\quad+ C  \tau^2 \int_{t_{0}}^{t_{\ell}} \left[ \norm{\partial_s\phi(s)}{L^2}^2 + \norm{\partial_{ss}\phi(s)}{L^2}^2   \right] ds  + C \int_{t_{0}}^{t_{\ell}} \norm{\partial_s \phi(s) - P_h \partial_s \phi(s)}{2,h}^2 ds .
\end{align*}
Requiring $\beta > 1$ and $\epsilon <\min\{1 + \frac{(1-\gamma)C_{coer} - 1}{C_{coer}},1 + 4 - \frac{4(1-\gamma)}{C_{coer}}\} $, we have
\begin{align*}
&C_1 \norm{\errh^{\phi,\ell}}{2, h}^2 + C \norm{\errh^{\phi,\ell}}{L^2}^2 + C \tau \sum\limits_{m=1}^{\ell}\norm{\nabla \errh^{\mu,m}}{L^2}^2  
\nonumber
\\
&\quad+ C \tau^2 \sum\limits_{m=1}^{\ell} \left[ C_{coer} \norm{\dtau \errh^{\phi,\ell}}{2, h}^2 + (5-\epsilon)\norm{\dtau \errh^{\phi,m}}{L^2}^2 + \norm{\nabla \dtau \errh^{\phi,m}}{L^2}^2\right] 
\nonumber
\\
&\le\,  C_2 \tau \sum\limits_{m=1}^{\ell }  \norm{\errh^{\phi,m}}{2,h}^2 + C \norm{\errP^{\phi,\ell}}{2,h}^2 + C \left[\text{Osc}_j(\mu^\ell)\right]^2 
\\
&\quad+C \tau \sum\limits_{m=1}^{\ell}\left[ \norm{\nabla \errR^{\mu,m}}{L^2}^2 +  \norm{\errP^{\phi,m}}{2,h}^2  + \left[\text{Osc}_j(\partial_t \mu(t^*))\right]^2 \right]
\\
&\quad+ C  \tau^2 \int_{t_{0}}^{t_{\ell}} \left[ \norm{\partial_s\phi(s)}{L^2}^2 + \norm{\partial_{ss}\phi(s)}{L^2}^2   \right] ds  + C \int_{t_{0}}^{t_{\ell}} \norm{\partial_s \phi(s) - P_h \partial_s \phi(s)}{2,h}^2 ds 
\nonumber
\\
&\le\,  C_2 \tau \sum\limits_{m=1}^{\ell }  \norm{\errh^{\phi,m}}{2,h}^2 + C \norm{\errP^{\phi,\ell}}{2,h}^2 + C \left[\text{Osc}_j(\mu^\ell)\right]^2 
\\
&\quad+C \tau \sum\limits_{m=1}^{\ell}\left[ C h^2 |\mu|_{H^2(\Omega)}^2 +  \norm{\errP^{\phi,m}}{2,h}^2  + \left[\text{Osc}_j(\partial_t \mu(t^*))\right]^2 \right]
\\
&\quad+ C  \tau^2 \int_{t_{0}}^{t_{\ell}} \left[ \norm{\partial_s\phi(s)}{L^2}^2 + \norm{\partial_{ss}\phi(s)}{L^2}^2   \right] ds  + C \int_{t_{0}}^{t_{\ell}} \norm{\partial_s \phi(s) - P_h \partial_s \phi(s)}{2,h}^2 ds ,
\end{align*}
where we have used well-known properties of the Ritz projection operator \eqref{eq:H1-Ritz-projection} in the last step. Combining like terms and considering the higher regularities \eqref{eq:higher-regularities} and the fact that $\text{Osc}_j(f) \le Ch^2$ for some function $f \in L^2(\Omega)$, we have
\begin{align*}
&\norm{\errh^{\phi,\ell}}{2, h}^2 + C \norm{\errh^{\phi,\ell}}{L^2}^2 + C \tau \sum\limits_{m=1}^{\ell}\norm{\nabla \errh^{\mu,m}}{L^2}^2  
\nonumber
\\
&\quad+ C \tau^2 \sum\limits_{m=1}^{\ell} \left[ C_{coer} \norm{\dtau \errh^{\phi,\ell}}{2, h}^2 + (5-\epsilon)\norm{\dtau \errh^{\phi,m}}{L^2}^2 + \norm{\nabla \dtau \errh^{\phi,m}}{L^2}^2\right] 
\nonumber
\\
&\le \, \frac{C_2\tau}{C_1 - C_2 \tau} \sum\limits_{m=1}^{\ell-1 }  \norm{\errh^{\phi,m}}{2,h}^2 + C \norm{\errP^{\phi,\ell}}{2,h}^2 +C \tau \sum\limits_{m=1}^{\ell} \norm{\errP^{\phi,m}}{2,h}^2  
\\
&\quad+ C \int_{t_{0}}^{t_{\ell}} \norm{\partial_s \phi(s) - P_h \partial_s \phi(s)}{2,h}^2 ds + C (T+1) h^2  + C \tau^2 .
\end{align*}

Allowing for $0 \le \tau \le \tau_0$ such that $\tau_0 := \frac{C_1}{C_2}$,  noting the higher regularities \eqref{eq:higher-regularities}, and using the Ritz projection properties from Appendix \ref{app:H2-Ritz-proj-properties}, we have
\begin{align*}
&\norm{\errh^{\phi,\ell}}{2, h}^2 + C \norm{\errh^{\phi,\ell}}{L^2}^2 + C \tau \sum\limits_{m=1}^{\ell}\norm{\nabla \errh^{\mu,m}}{L^2}^2  
\nonumber
\\
&\quad+ C \tau^2 \sum\limits_{m=1}^{\ell} \left[ C_{coer} \norm{\dtau \errh^{\phi,\ell}}{2, h}^2 + (5-\epsilon)\norm{\dtau \errh^{\phi,m}}{L^2}^2 + \norm{\nabla \dtau \errh^{\phi,m}}{L^2}^2\right] 
\nonumber
\\
&\le \, C_3 \tau \sum\limits_{m=1}^{\ell-1 }  \norm{\errh^{\phi,m}}{2,h}^2 + C ((T+1) h^2 + \tau^2 ) ,
\end{align*}
where none of the constants above depend on the mesh size $h$ or the time step size $\tau$.  Applying a discrete Gr\"{o}nwall's concludes the proof.
\end{proof}

\begin{rmk}
Again following \cite{brenner:11:frontiers}, we note that $C_{coer}$ can be chosen to be close to 1 as long as the penalty parameter $\alpha$ is large enough. In this case, $\gamma$ could also be chosen close to 0 and \eqref{eq:error-bound-final} will hold as long as $\epsilon < 1$. 
\end{rmk}


\section{Numerical Experiments}\label{sec:numerical}

In this section, we present two numerical experiments demonstrating the effectiveness of our method. All numerical experiments are completed using the FEniCS project \cite{fenics}. In the first experiment, we show that our method converges with first order accuracy with regard to both time and space. We futhermore show that the discrete energy \eqref{eq:discrete-energy} dissipates over time and we benchmark our results against those found in the paper by Hu, Wise, Wang, and Lowengrub \cite{HWWL:09:PFC}. Therefore, following \cite{HWWL:09:PFC}, we set the initial conditions to be 
\begin{align*}
\phi(x,y) &= 0.07 - 0.02 \cos\left(\frac{2\pi(x-12)}{32}\right) \sin\left(\frac{2\pi(y-1)}{32}\right) 
\\
&+ 0.02\cos^2\left(\frac{\pi(x+10)}{32}\right) \cos^2\left(\frac{\pi(y+3)}{32}\right) - 0.01 \sin^2\left(\frac{4\pi x}{32}\right) \sin^2\left(\frac{4\pi (y-6)}{32}\right)
\end{align*}
and solve on the domain $\Omega = (0,32) \times (0,32)$ to a final stopping time of $T=10$. We solve using the mesh sizes shown in the table below and scale the time step size with the mesh size via $\tau = 0.05 h$. We set $\epsilon = 0.025$ and the penalty parameter $\alpha = 20$. We point out that Neumann boundary conditions are implemented and the finite element spaces are the $P_2, P_1$ Lagrange finite element spaces, respectively. To show first order convergence in the energy norm, we assign the solution from a mesh size of $h=\nicefrac{32}{512}$ with $\tau = 0.05h$ and $T=10$ as the `exact' solution, $\phi_{exact}$. We then define $error_\phi := \phi_h - \phi_{exact}$, where $\phi_h$ indicates the solution on the mesh size $h$ with $\tau = 0.05h$ and $T=10$. We use a similar strategy to compute the errors with respect to $\mu$. Table \ref{tab:convergence} shows the errors and rates of convergence given the parameters noted in the text above.

In Figure \ref{fig:scaled-energy}, the time evolution of the scaled total energy $\nicefrac{F}{32^2}$ is shown using the initial conditions stated above, a mesh size of $h=\nicefrac{32}{256}$ and a time step size of $\tau = 0.05h$ with all other parameters defined above. We note that the scaled total energy shown here almost exactly matches that shown in Figure 1 of \cite{HWWL:09:PFC} where a second order in time finite difference scheme was used to approximate solutions to the PFC equation considering all the same parameter values and the chosen initial conditions. Figure \ref{fig:solutions} displays the initial conditions specified above with a mesh size of $h=\nicefrac{32}{256}$ on the left and the solution at the final stopping time of $T=10$. Again, comparing these figures to those found in Figure 1 of \cite{HWWL:09:PFC}, we see that our method produces the expected results. We remark that the chosen mesh sizes are fairly coarse due to the chosen size of the domain and the fact that finer mesh sizes would require large computational costs.  However, the domain was chosen as in \cite{HWWL:09:PFC} in order to benchmark our method. Finer mesh sizes will be considered as part of future work on building an efficient solver.

The purpose of the second numerical experiment is to demonstrate that our method accurately captures grain growth of a polycrystal in a supercooled liquid. For the initial conditions, we define three crystallites with different orientations as in \cite{GN:2012:PFC}. The computational domain for this example is $\Omega = [0, 201], \epsilon = 0.25, \tau = 1, h = 0.5,$ and $\alpha = 20$. Snapshots of the numerical solution are shown at different times in Figure \ref{fig:grain-growth}. We observe the growth of distinct crystallites and remark that well-defined crystal-liquid interfaces are clearly observed. Similar results were observed in \cite{GN:2012:PFC, HWWL:09:PFC}.

As a final numerical experiment, we present the total scaled energy evolution for time step sizes $\tau=h, \ 5h, 10h$ with $h=\nicefrac{32}{256}$ in Figure~\ref{fig:scaled-energydt10-5-1}. The large time step sizes have been chosen to emphasize unconditional stability. As observed in Figure~\ref{fig:scaled-energy}, the energy curves decay for all time step sizes thereby demonstrating the unconditional stability of the scheme.

\section{Conclusion}\label{sec:conclusion}

In this paper, we have developed a C$^0$ interior penalty finite element method to solve a special case of the phase field crystal equation \eqref{eq:pfc-pre}. We were able to demonstrate that our method is uniquely solvable, unconditionally energy stable, and unconditionally convergent. We were also able to demonstrate that our method benchmarks well against numerical experiments established in the existing literature. Future work includes the extension of the C$^0$ interior penalty method developed in Section \ref{sec:fem} to the case in which a non-constant mobility is considered and the case in which periodic boundary conditions are considered. Regarding the case in which a non-constant mobility is considered, it is believed that unconditional solvability and stability could be achieved as seen in \cite{HWWL:09:PFC} with the biggest outstanding question centered around the error analysis. Additionally, application of the method to related models such as the modified phase field crystal equation and building efficient solvers for these methods remains of interest.

\section*{Acknowledgment}
We would like to thank Steven M.~Wise for his valuable advice regarding the phase field crystal model. We would also like to thank Susanne C.~Brenner and Li-Yeng Sung for their valuable advice regarding the C$^0$ interior penalty method. The second author would like to acknowledge the support of NSF Grant No. DMS-1520862.


\bibliographystyle{plain}
\bibliography{PFC_C0IP}


\newpage

\begin{table}[H]
	\centering
\begin{tabular}{| c | c | c | c | c |} 
\hline
$h$ & $\norm{error_{\phi}}{2,h}$ & rate & $\norm{error_{\mu}}{H^1}$ & rate \\
\hline
$\nicefrac{32}{8}$ & 0.08412 & N/A & 0.00522 & N/A \\ 
$\nicefrac{32}{16}$ & 0.05896 & 0.71329 & 0.00242 & 1.07627 \\ 
$\nicefrac{32}{32}$ & 0.03466 & 0.85058 & 0.00157 & 0.76970 \\ 
$\nicefrac{32}{64}$ & 0.01568 & 1.10514 & 0.00103 & 0.76082 \\ 
$\nicefrac{32}{128}$ & 0.00601 & 1.30482 & 0.00041 & 1.25840 \\ 
$\nicefrac{32}{256}$ & 0.00255 & 1.17707 & 0.00016 & 1.27362 \\ 
\hline
\end{tabular}
\caption{Errors and convergence rates of the C$^0$-IP method. Parameters and initial conditions are given in the text.}
\label{tab:convergence}
\end{table}

\newpage

\begin{figure}[H]
\centering
\includegraphics[scale=.5]{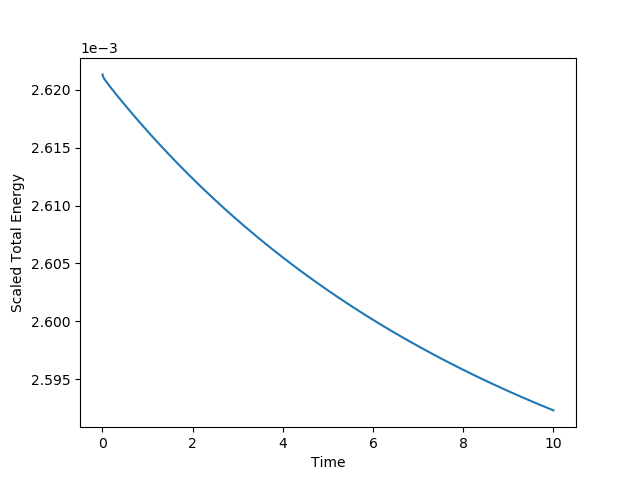}
\caption{The time evolution of the scaled total energy $\nicefrac{F}{32^2}$. The mesh size is $h=\nicefrac{32}{256}$ and the time step size is $\tau = 0.05h$. All other parameters are defined in the text.}
\label{fig:scaled-energy}
\end{figure}

\begin{figure}[H]
\centering
\subfloat{\includegraphics[width = 3in]{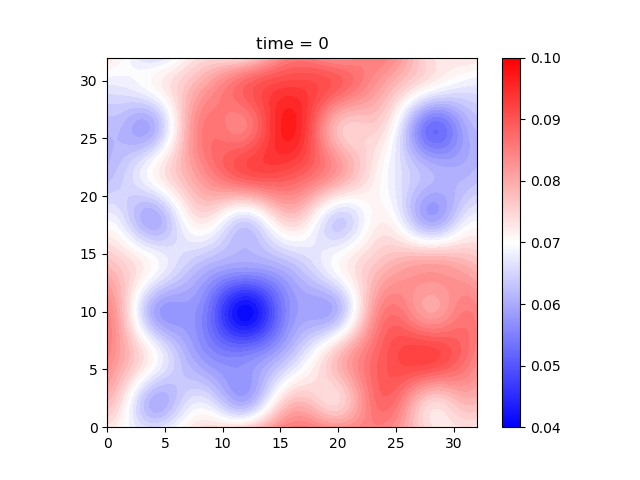}}
\subfloat{\includegraphics[width = 3in]{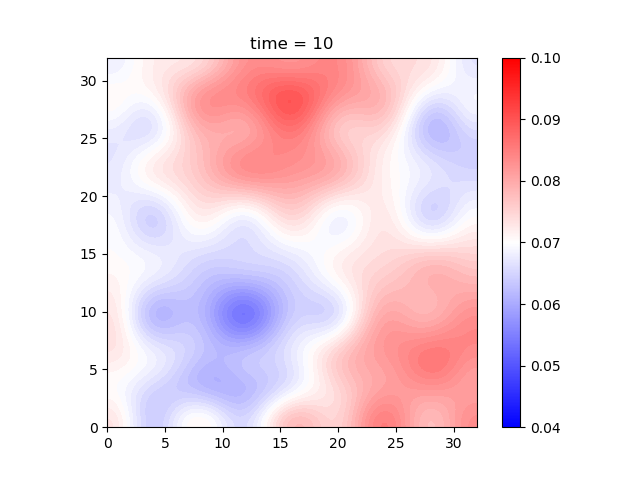}}
\caption{Density plots in which the white regions indicate $\phi = 0.0685$, the red region indicates $\phi = 0.097 $, and the blue region indicates $\phi = 0.04$. The initial configuration is shown on the left and the solution at time $T=10$ is shown on the right. The mesh size is $h=\nicefrac{32}{256}$ and the time step size is $\tau = 0.05h$. All other parameters are defined in the text.}
\label{fig:solutions}
\end{figure}

\begin{figure}[h!]
\centering
\subfloat{\includegraphics[width = 2in]{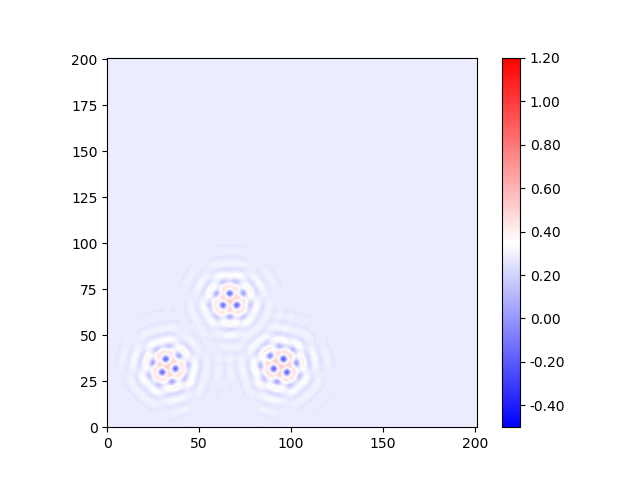}}
\subfloat{\includegraphics[width = 2in]{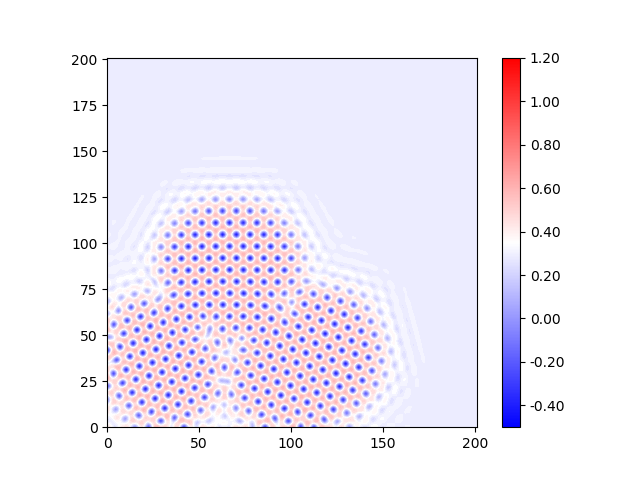}}
\subfloat{\includegraphics[width = 2in]{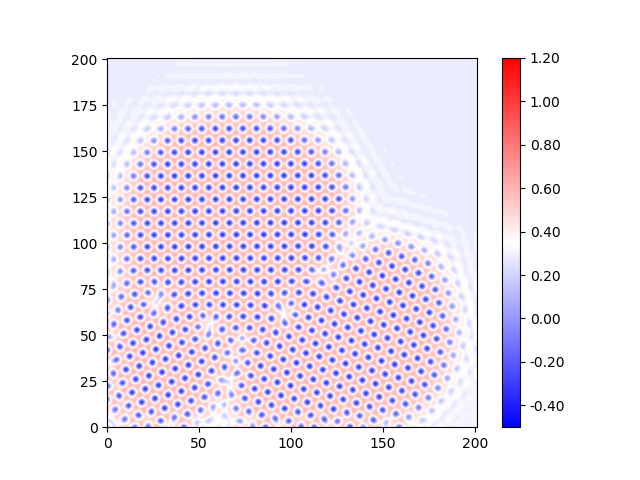}}
\\
\subfloat{\includegraphics[width = 2in]{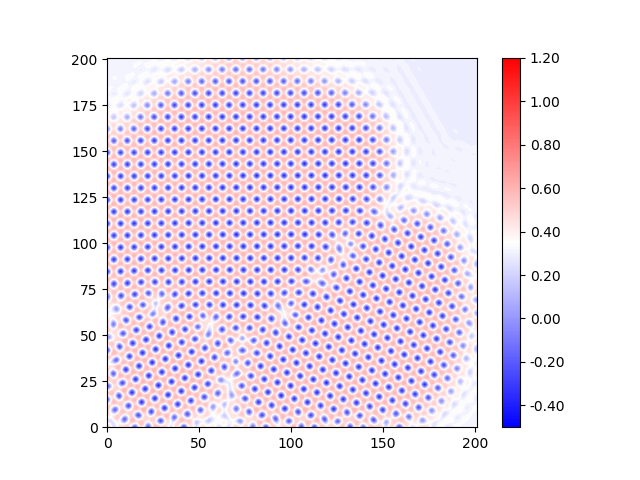}}
\subfloat{\includegraphics[width = 2in]{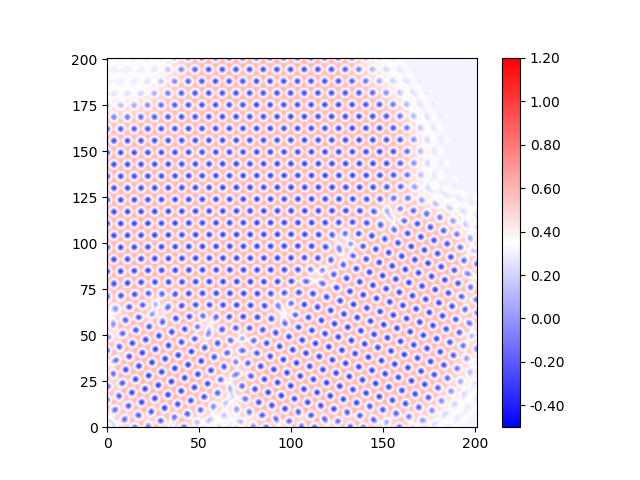}}
\subfloat{\includegraphics[width = 2in]{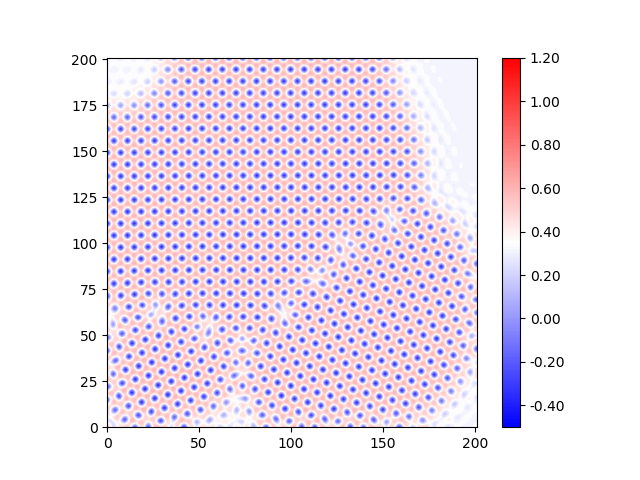}}
\caption{Snapshots of grain growth at times $T=100, 1000, 2000, 3000, 4000, 5000$ are shown above. The mesh size is $h=\nicefrac{201}{402}$ and the time step size is $\tau = 1$. All other parameters are defined in the text.}
\label{fig:grain-growth}
\end{figure}

\begin{figure}[H]
	\centering
			\includegraphics[scale=.45]{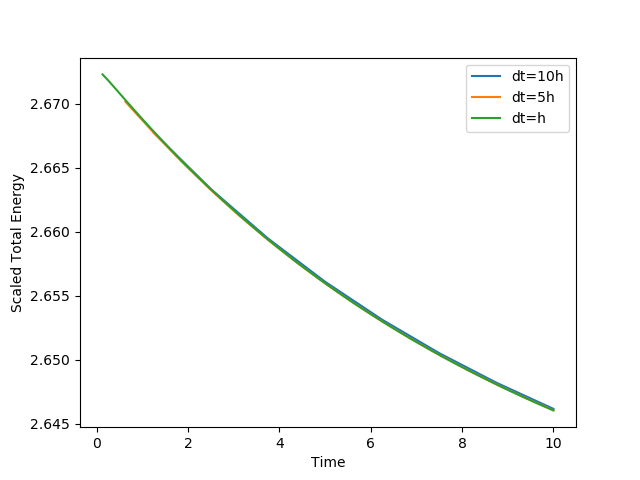} \quad
				\includegraphics[scale=.45]{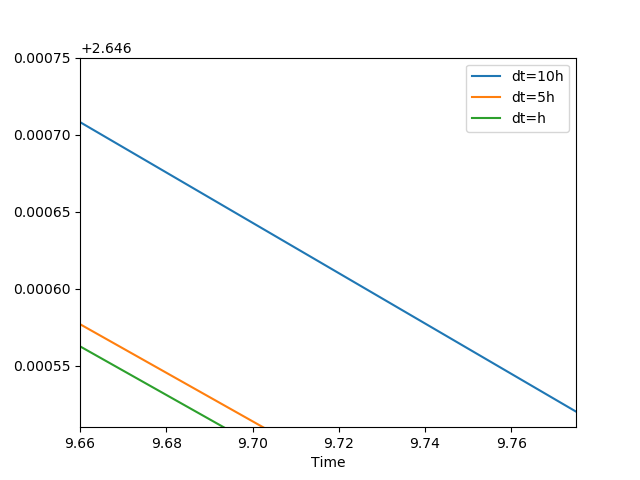}
	\caption {Unconditional stability demonstrated through the time evolution of the scaled total energy $\nicefrac{F}{32^2}$ for time step sizes $\tau = 10h, \ 5h, h$ with the spacial step size $h = \nicefrac{32}{256}$ on the \emph{(left)} and a zoomed image on the \emph{(right)}. All other parameters are defined in the text.}
	\label{fig:scaled-energydt10-5-1}
\end{figure}

\newpage

\appendix

\section{Proof of Lemma \ref{thm:pfc-uniqueness}}\label{app:uniqueness}

	\begin{proof}
We begin by showing $G_h$ is strictly convex. To do so, we consider the second derivative of $G_h(\varphih + s \psi_h)$ with respect to $s$ and set $s=0$. Hence,
	\begin{align*}
G_h(\varphih + s\psi_h) =& \, \frac{\tau}{2}\norm{\frac{\varphih + s\psi_h -\varphih^{m-1}}{\tau}}{-1,h}^2 + \frac{1}{2}\aIPh{\varphih + s\psi_h}{\varphih + s\psi_h} 
	\nonumber
	\\
&+\frac{1}{4}\norm{\varphih + s\psi_h + \avephio}{L^4}^4  +\frac{1-\epsilon}{2}\norm{\varphih + s\psi_h + \avephio}{L^2}^2 
\nonumber
\\
&- 2 \iprd{\nabla \varphih^{m-1}}{\nabla (\varphih + s\psi_h)} .
	\end{align*}
Taking the derivative with respect to $s$, we have
	\begin{align}
G_h^{\prime} (\varphih + s \psi_h) =& \, \frac{1}{\tau} \left(\varphih + s \psi_h -\varphih^{m-1}, \psi_h \right)_{-1,h} + \aIPh{\varphih + s\psi_h}{\psi_h} 
	\nonumber
	\\
&+ \iprd{\psi_h \left(\varphih + s \psi_h + \avephio \right)}{\left(\varphih + s \psi_h + \avephio \right)^2} + (1-\epsilon)\iprd{\varphih + s \psi_h}{ \psi_h} 
\nonumber
\\
&- 2\iprd{\nabla \varphih^{m-1}}{\nabla \psi_h} ,
	\label{eq:pfc-G-first-derivative}
	\end{align}
where $(\zeta, \xi)_{-1,h} := \iprd{\zeta}{\mathsf{T}_h \xi} $. Taking the second derivative with respect to $s$, we have
	\begin{align*}
G_h^{\prime \prime} (\varphih + s \psi_h) = \frac{1}{\tau} \norm{\psi_h}{-1,h}^2 + \aIPh{\psi_h}{\psi_h} + 3\iprd{\left(\varphih + s \psi_h + \avephio \right)^2}{\psi_h ^2} + (1-\epsilon)\norm{\psi_h}{L^2}^2 .
	\end{align*}
Setting $s=0$ and using the coercivity of $a^{IP}_h(\cdot,\cdot)$, we have
	\begin{align*}
G_h^{\prime \prime} (\varphih) = \frac{1}{\tau} \norm{\psi_h}{-1,h}^2 + 3\iprd{\left(\varphih + \avephio \right)^2}{\psi_h ^2} + (1-\epsilon)\norm{\psi_h}{L^2}^2 + \aIPh{\psi_h}{\psi_h}> 0
	\end{align*}
for all $\epsilon < 1$ and $\varphih \in \Zoh$. 

To show $G_h$ is coercive, we need to show that there exists constants $\gamma >0, \beta \ge 0$ such that $G_h(\varphih) \ge \gamma \norm{\varphih}{2,h} - \beta$ for all $\varphih \in \Zoh$. Using the Cauchy Schwarz inequality, Young's inequality, and a Poincar\'{e} type inequality \cite{BWZ:04:poincare, nevcasmethodes}, we have
	\begin{align*}
G_h(\varphih) &\ge \frac{1}{2}\aIPh{\varphih}{\varphih} -2 \iprd{\nabla \varphih^{m-1}}{\nabla \varphih}
\\
&\ge \frac{C_{coer}}{2} \norm{\varphih}{2,h}^2 - 2 \norm{\nabla \varphih^{m-1}}{L^2} \norm{\nabla \varphih}{L^2}
\\
&\ge \frac{C_{coer}}{2} \norm{\varphih}{2,h}^2 - C_P \norm{\nabla \varphih^{m-1}}{L^2} \norm{\varphih}{2,h}
\\
&\ge \frac{C_{coer}}{2}  \norm{\varphih}{2,h}^2 - \frac{C_P}{C_{coer}} \norm{\nabla \varphih^{m-1}}{L^2}^2 - \frac{C_{coer}}{4}\norm{\varphih}{2,h}^2
\\
&\ge \frac{C_{coer}}{4} \norm{\varphih}{2,h}^2 - \frac{C_P}{C_{coer}} \norm{\nabla \varphih^{m-1}}{L^2}^2,
	\end{align*}
where $C_{coer}$ depends on the coercivity of the $a^{IP}_h(\cdot, \cdot)$ inner product and $C_P$ depends on the Poincar\'{e} type inequality. Therefore,
	\begin{align*}
G_h(\varphih) \ge \gamma \norm{\varphih}{2,h}^2 - \beta,
	\end{align*}	 
where $\gamma = \frac{C_{coer}}{4}$ and $\beta = \frac{C_P}{C_{coer}} \norm{\nabla \varphih^{m-1}}{L^2}^2$ do not depend on $\varphih$. Hence, $G_h$ has a unique minimizer, $\varphih^m \in \Zoh$ which solves
	\begin{align*}
G_h^{\prime} (\varphih^m) =& \, \frac{1}{\tau} \left(\varphih^m -\varphih^{m-1},  \psi_h \right)_{-1,h} + \aIPh{\varphih^m}{\psi_h} + \iprd{\left(\varphih^m + \avephio \right)^3}{\psi_h}  
\\
&+ (1-\epsilon)\iprd{\varphih^m + \avephio}{\psi_h}  - \iprd{\nabla \varphih^{m-1}}{\nabla \psi_h} =0,
	\end{align*} 
for all $\psi_h \in \Zoh$ where we have set $s=0$ in \eqref{eq:pfc-G-first-derivative}. Therefore, $\varphih^m \in \Zoh$ is the unique minimizer of $G_h$ if and only if it is the unique solution to
	\begin{equation*}
\iprd{\mu_{h,\star}^m}{\psi_h} - \aIPh{\varphih^m}{\psi_h} - \iprd{\left(\varphih^m + \avephio\right)^3}{\psi_h} - (1-\epsilon)\iprd{\varphih^m + \avephio}{\psi_h} +2 \iprd{ \nabla \varphih^{m-1}}{ \nabla \psi_h }  =0
	\end{equation*}
for all $\psi_h \in \Zoh$, where $\mu_{h,\star}^m \in \Voh$ is the unique solution to
	\begin{align*}
\iprd{\nabla \mu_{h,\star}^m}{\nabla \nu_h}  = -\iprd{\frac{\varphih^m-\varphih^{m-1}}{\tau}}{\nu_h} &  \qquad \forall \,  \nu_h \in\Voh.
	\end{align*}

	\end{proof}

\section{Proof of Lemma \ref{lem:minus-1-norm-bound}} \label{app:-1,h_bound_proof}

\begin{proof}
First, we note that for all $\chi\in Z_h$ and all $\zeta\in\Zoh$ \cite{DFW:15:CHDS}, 
	\begin{equation}
\left|\iprd{\zeta}{\chi}\right| \le \norm{\zeta}{-1,h} \norm{\nabla\chi}{L^2} 
	\label{eq:plus-1-minus-1-estimate}
	\end{equation}
and similarly for all $g \in Z$ and all $\zeta\in\Zoh$, 
	\begin{equation}
\left|\iprd{\zeta}{g}\right| \le \norm{\zeta}{-1,h} \norm{\nabla g}{L^2} .
	\label{eq:plus-1-minus-1-estimate-g}
	\end{equation}
Setting $\nuh = \mathsf{T}_h \dtau \errh^{\phi,m}$ in \eqref{eq:error-eq-a}, using properties \eqref{eq:plus-1-minus-1-estimate} and \eqref{eq:plus-1-minus-1-estimate-g} above, Taylor's theorem, and Young's and H\"{o}lder's inequalities, we have
\begin{align*}
\norm{\dtau \errh^{\phi,m}}{-1,h}^2 &= -\iprd{\nabla \errh^{\mu,m}}{\nabla \mathsf{T}_h \dtau \errh^{\phi,m}} + \iprd{\dtau \phi^{m} - \partial_t \phi^{m}}{\mathsf{T}_h \dtau \errh^{\phi,m}} - \iprd{\dtau \errP^{\phi,m}}{\mathsf{T}_h \dtau \errh^{\phi,m}}
\\
&\le  \norm{\nabla \errh^{\mu,m}}{L^2}\norm{\nabla \mathsf{T}_h \dtau \errh^{\phi,m}}{L^2} + \norm{\dtau \phi^{m} - \partial_t \phi^{m}}{L^2}\norm{\mathsf{T}_h \dtau \errh^{\phi,m}}{L^2}  
\\
&\quad+ \norm{ \dtau \errP^{\phi,m}}{L^2}\norm{\mathsf{T}_h \dtau \errh^{\phi,m}}{L^2} 
\\
& \le  \norm{\nabla \errh^{\mu,m}}{L^2}\norm{\nabla \mathsf{T}_h \dtau \errh^{\phi,m}}{L^2} + C \norm{\dtau \phi^{m} - \partial_t \phi^{m}}{L^2}\norm{\nabla \mathsf{T}_h \dtau \errh^{\phi,m}}{L^2}  
\\
&\quad+ C \norm{ \dtau \errP^{\phi,m}}{L^2}\norm{\nabla \mathsf{T}_h \dtau \errh^{\phi,m}}{L^2} 
\\
&\le  \norm{\nabla \errh^{\mu,m}}{L^2}^2 + \frac{3}{4} \norm{\dtau \errh^{\phi,m}}{-1,h}^2 + C \norm{\dtau \phi^{m} - \partial_t \phi^{m}}{L^2}^2 + C \norm{\dtau \errP^{\phi,m}}{L^2}^2 
\\
&\le \norm{\nabla \errh^{\mu,m}}{L^2}^2 + \frac{3}{4} \norm{\dtau \errh^{\phi,m}}{-1,h}^2 + C \tau \int_{t_{m-1}}^{t_{m}}\norm{\partial_{ss}\phi(s)}{L^2}^2 ds 
\\
&\quad+ \frac{C}{\tau} \int_{t_{m-1}}^{t_{m}} \norm{ \partial_s \phi(s) -\partial_s P_h \phi(s) }{2,h}^2 ds,
\end{align*}
where we note that the inequality $\displaystyle \norm{ \dtau \errP^{\phi,m}}{L^2}^2  \le  \frac{C}{\tau} \int_{t_{m-1}}^{t_{m}} \norm{ \partial_s \phi(s) -\partial_s P_h \phi(s) }{2,h}^2 ds$ can be shown using Taylor's theorem with integral remainder term and an application of the Cauchy-Schwarz and Poincar\'{e} inequalities as follows
\begin{align*}
\norm{\dtau \errP^{\phi,m}}{L^2}^2 &= \norm{\frac{1}{\tau}\int_{t_{m-1}}^{t_m}  \left(\partial_s \left(\phi(s) - P_h \phi(s)\right) \right)\, ds}{L^2}^2
\\
&= \frac{1}{\tau^2} \norm{\int_{t_{m-1}}^{t_m}  \left(\partial_s \phi(s) - \partial_s P_h \phi(s) \right)\, ds}{L^2}^2
\\
&\le \frac{1}{\tau^2} \int_\Omega \left( \left(\int_{t_{m-1}}^{t_m}  1^2 \, ds \right)^{\nicefrac{1}{2}} \left(\int_{t_{m-1}}^{t_m} \left( \partial_s \phi(s) -  \partial_s P_h \phi(s) \right)^2\, ds \right)^{\nicefrac{1}{2}} \right)^2 d{\bf x}
\\
&= \frac{1}{\tau} \int_{t_{m-1}}^{t_m} \int_\Omega \left(\partial_s \phi(s) - \partial_s P_h \phi(s) \right)^2 d{\bf x}\, ds 
\\
&= \frac{C}{\tau} \int_{t_{m-1}}^{t_{m}} \norm{\partial_s \phi(s) -\partial_s P_h \phi(s) }{L^2}^2 ds
\\
&\le \frac{C}{\tau} \int_{t_{m-1}}^{t_{m}} \norm{ \partial_s \phi(s) -\partial_s P_h \phi(s)}{2,h}^2 ds.
\end{align*}
The result immediately follows.
\end{proof}
	

\newpage

\section{Details from the Proof of Lemma \ref{lem:bounds-oscillations}}\label{app:osc-bound-proof}

Following the medius analysis presented in \cite{brenner:11:frontiers}(see pages 101-106), we proceed by bounding all but the last three terms on the right-hand side:
\begin{align*}
&\left|\sum_{K \in \TT_h}  \int_K \left(\Delta^2 P_h \phi^m + \left(P_h \phi^m\right)^3+(1-\epsilon)P_h \phi^m +2 \Delta P_h \phi^{m}- \mu^m\right) \left(\errh^{\phi,m} -E_h \errh^{\phi,m}\right) \, dx \right| 
\\
&\qquad\le \left( \sum_{K \in \TT_h}  \int_K h^4 \norm{\Delta^2 P_h \phi^m + \left(P_h \phi^m\right)^3+(1-\epsilon)P_h \phi^m + 2\Delta P_h \phi^{m}- \mu^m}{L^2(K)}^2\right)^{\nicefrac{1}{2}} \norm{\errh^{\phi,m}}{2,h},
\\
& \left|\sum_{K \in \TT_h}  \int_K \nabla^2 (P_h \phi^m - \phi^m) : \nabla^2 (E_h \errh^{\phi,m}) dx \right| \le C \left(\sum_{K \in \TT_h} |\phi^m - P_h \phi^m|_{H^2(K)}^2\right)^{\nicefrac{1}{2}} \norm{\errh^{\phi,m}}{2,h} 
\\
&\left|\sum_{e \in \Eh} \int_e \dgal[\Bigg]{ \frac{\partial^2 \errh^{\phi,m} }{\partial n_e^2} } \left\llbracket \frac{\partial P_h \phi^{m}}{\partial n_e} \right\rrbracket \right| 
\\
&\qquad\le C \left( \sum_{e \in \Eh} \frac{1}{|e|} \norm{\left\llbracket \frac{\partial (P_h \phi^{m} - \phi^m)}{\partial n_e} \right\rrbracket}{L^2(e)}^2 \right)^{\nicefrac{1}{2}} \norm{\errh^{\phi,m}}{2,h},
\\
&\left|\sum_{e \in \Eh} \int_e \left\llbracket \frac{\partial \Delta P_h \phi^{m}}{\partial n_e} \right\rrbracket \left(\errh^{\phi,m} -E_h \errh^{\phi,m}\right) dS \right| 
\\
&\qquad\le C \left( \sum_{e \in \Eh} |e|^3 \norm{\left\llbracket \frac{\partial \Delta P_h \phi^{m}}{\partial n_e} \right\rrbracket }{L^2(e)}^2\right)^{\nicefrac{1}{2}} \norm{\errh^{\phi,m}}{2,h},
\\
&\left|\sum_{e \in \Eh} \int_e  \left\llbracket \frac{\partial^2 P_h \phi^m}{\partial n_e^2} \right\rrbracket \dgal[\Bigg]{\frac{\partial  \left(\errh^{\phi,m} -E_h \errh^{\phi,m}\right)}{\partial n_e} } dS \right| 
\\
&\qquad \le C \left( \sum_{e \in \Eh} |e| \norm{\left\llbracket \frac{\partial^2 P_h \phi^{m}}{\partial n_e^2} \right\rrbracket }{L^2(e)}^2\right)^{\nicefrac{1}{2}} \norm{\errh^{\phi,m}}{2,h},
\\
&\left| \sum_{e \in \Eh} \int_e \left\llbracket \frac{\partial^2 P_h \phi^m}{\partial n_e \partial t_e} \right\rrbracket \frac{\partial  \left(\errh^{\phi,m} -E_h \errh^{\phi,m}\right)}{\partial t_e} dS \right| 
\\
&\qquad\le C \left( \sum_{e \in \Eh} \frac{1}{|e|} \norm{\left\llbracket \frac{\partial (P_h \phi^{m} - \phi^m)}{\partial n_e} \right\rrbracket }{L^2(e)}^2\right)^{\nicefrac{1}{2}} \norm{\errh^{\phi,m}}{2,h} ,
\\
&\left|  \alpha \sum_{e \in \Eh} \frac{1}{|e|} \int_e  \left\llbracket \frac{\partial P_h \phi^m}{\partial n_e} \right\rrbracket  \left\llbracket \frac{\partial \left(\errh^{\phi,m} -E_h \errh^{\phi,m}\right)}{\partial n_e} \right\rrbracket dS \right|
\\
&\qquad\le  C \left( \sum_{e \in \Eh} \frac{1}{|e|} \norm{\left\llbracket \frac{\partial (P_h \phi^{m} - \phi^m)}{\partial n_e} \right\rrbracket }{L^2(e)}^2\right)^{\nicefrac{1}{2}} \norm{\errh^{\phi,m}}{2,h} .
\end{align*}


\newpage

\section{Ritz Projection Properties}\label{app:H2-Ritz-proj-properties}

\begin{lemma}
Let $\epsilon < 1$ and $u \in H^3(\Omega)$  solve the model problem
\begin{align*}
\Delta^2 u + (1-\epsilon) u &= f \quad \text{in} \quad \Omega
\\
\frac{\partial u}{\partial n} = \frac{\partial \Delta u}{\partial n} &=0 \quad \text{on} \quad \partial\Omega,
\end{align*}
where $f \in L^2(\Omega)$ and define the following norm on the space $H^3(\Omega, \TT_h)$:
\begin{align*}
||| v |||_h^2 := \sum\limits_{K \in \TT_h} |v|_{H^2(K)}^2 + \alpha \sum_{e \in \Eh} |e|^{-1} \norm{\left\llbracket \frac{\partial v}{\partial n}\right\rrbracket}{L^2(e)}^2 + \frac{1}{\alpha} \sum_{e \in \Eh} |e| \norm{\dgal[\Bigg]{\frac{\partial^2 v }{\partial n_e^2} }  }{L^2(e)}^2. 
\end{align*}
Then the Ritz projection \eqref{eq:H2-Ritz-projection} satisfies the following:
\begin{align*}
\norm{u - P_h u}{2,h} \le C h \norm{u}{H^3(\Omega)}.
\end{align*}

\begin{proof}
Note that on the finite element space $Z_h$, the two norms $\norm{\cdot}{2,h}$ and $||| \cdot |||_h$ are equivalent. Additionally, following \cite{brenner:11:frontiers}, it can be shown that there exists constants $C_1$ and $C_2$ such that
\begin{align*}
C_1 ||| u |||_h^2 \le \aIPh{u}{u} \text{ for all } u \in H^3(\Omega, \TT_h)
\end{align*}
and 
\begin{align*}
\aIPh{u}{w} \le C_2 ||| u |||_h ||| w |||_h  \text{ for all } u, w \in H^3(\Omega, \TT_h).
\end{align*}
Thus, using the definition of the Ritz projection \eqref{eq:H2-Ritz-projection}, letting $v_h \in Z_h$, and using the Cauchy-Schwarz and Young's inequalities, we have
\begin{align*}
||| P_h u - u |||_h^2 + (1-\epsilon) \norm{P_h u - u}{L^2}^2 &\le C_2 \aIPh{P_h u - u}{P_h u - u} 
\\
&\quad+ (1 - \epsilon) \iprd{P_h u - u}{P_h u - u}
\\
&= C_2 \aIPh{P_h u - u}{v_h - u} + (1 - \epsilon) \iprd{P_h u - u}{v_h - u}
\\
&\le C_1 C_2 ||| P_h u - u |||_h ||| v_h  - u |||_h 
\\
&\quad+ (1 - \epsilon) \norm{P_h u - u}{L^2(\Omega)}\norm{v_h - u}{L^2(\Omega)}
\\
&\le \frac{1}{2} ||| P_h u - u |||_h  + \frac{C_1^2 C_2^2}{2}||| v_h - u |||_h^2
\\
&\quad+ \frac{(1 - \epsilon)}{2} \norm{P_h u - u}{L^2(\Omega)}^2  + \frac{(1 - \epsilon)}{2} \norm{v_h - u}{L^2(\Omega)}^2.
\end{align*}
Combining like terms and multiplying by 2, we have
\begin{align*}
||| P_h u - u |||_h^2 + (1-\epsilon) \norm{P_h u - u}{L^2}^2 &\le C_1^2 C_2^2 ||| v_h - u |||_h^2 + (1 - \epsilon) \norm{v_h - u}{L^2(\Omega)}^2 \, \forall v_h \in Z_h.
\end{align*}
Let $\Pi_h: C(\overline{\Omega}) \rightarrow Z_h$ be the Lagrange nodal interpolation operator. Then choosing $v_h = \Pi_h u$ and following \cite{brenner:11:frontiers}, we have
\begin{align*}
||| P_h u - u |||_h^2  &\le C h^2 \norm{u}{H^3(\Omega)}^2.
\end{align*}
Noting that $\norm{v}{2,h}^2 \le ||| v |||_h^2$ for any $v \in H^3(\Omega)$ concludes the proof.
\end{proof}
\end{lemma}

\end{document}